%% file: SwitchingIdentities.tex
\documentclass[reqno, a4paper, 10pt]{amsart}
\pdfoutput=1
\usepackage{amsmath, amssymb,amsthm, mathrsfs}
\usepackage{txfonts}
\usepackage{enumerate}
\usepackage{color}
\usepackage{bbm}

\usepackage{asymptote}
\usepackage{graphicx}
\usepackage{subcaption}
\usepackage{float}

\usepackage{xcolor}

\newtheorem{theorem}{Theorem}

\newtheorem{corollary}[theorem]{Corollary}

\newtheorem{lemma}[theorem]{Lemma}
\newtheorem{proposition}[theorem]{Proposition}

\theoremstyle{remark}
\newtheorem{remark}[theorem]{Remark}

\newcommand{\F}{\mathcal{F}}

\newcommand{\E}{\mathbb{E}}
\newcommand{\N}{\mathbb{N}}

\newcommand{\R}{\mathbb{R}}
\newcommand{\Z}{\mathbb{Z}}

\newcommand{\G}{\mathcal G}

\newcommand{\di}{\mathrm{d}}

\newcommand{\EE}{{\mathbb E}}

\newcommand{\1}{\mathbbm 1}

\newcommand{\e}{\varepsilon}

\renewcommand{\epsilon}{\varepsilon}
\renewcommand{\phi}{\varphi}

\renewcommand{\P}{\mathbb{P}}

\numberwithin{equation}{section}
\numberwithin{theorem}{section}

\renewcommand{\subset}{\subseteq}

\renewcommand{\N}{\mathbb{N}}
\newcommand{\Ex}{\mathbb{E}^x} 
\newcommand{\Ey}{\mathbb{E}_y} 
\newcommand{\Exy}{\mathbb{E}^x_y} 
\newcommand{\Elambda}{\mathbb{E}^{\lambda}} 
\newcommand{\Elambday}{\mathbb{E}^{\lambda}_y}

\title{Switching identities by probabilistic means}

\begin{document}
\author{J. Backhoff-Veraguas}
\address{Institute for Mathematics, University of Vienna}
\email{julio.backhoff@univie.ac.at}
\author{A.M.G. Cox} 
\address{Department of Mathematical Sciences, University of Bath, U.K.}
\email{a.m.g.cox@bath.ac.uk}
\author{A. Grass}
\address{Institute for Mathematics, University of Vienna}
\email{annemarie.grass@univie.ac.at}
\author{M. Huesmann}
\address{Institute for Mathematical Stochastics, University of M\"unster, Germany}
\email{martin.huesmann@uni-muenster.de}

\thanks{AG has been funded by FWF Projects P28661 and Y782}
\thanks{MH has been funded by the Deutsche Forschungsgemeinschaft (DFG, German Research Foundation) under Germany's Excellence Strategy EXC 2044 –390685587, Mathematics M\"unster: Dynamics-–Geometry-–Structure.}

\date{\today}
  
\maketitle

{\bf Abstract: }
Switching identities have a long history in potential
theory and stochastic analysis. 
In recent work of Cox and Wang, a switching identity was 
used to connect an optimal stopping problem and
the Skorokhod embedding problem (SEP). 
Typically switching identies of
this form are derived using deep analytic connections. 
In this paper, we prove the switching identities 
using a simple probabilistic argument,
which furthermore highlights a previously 
unexplored symmetry between the Root and Rost
solutions to the SEP.

\medskip   

{\bf Keywords:} Skorokhod embedding, Root, Rost, Optimal stopping, Switching identities.

\begin{description}
\item [{\bf MSC 2020}] 60G40,  	60G42,  	60H30.
\end{description}

%

\section{Introduction}
Let $D$ be a rectangle of horizontal length $T$ and let $(0,x),(0,y)$ be points on the left boundary of $D$. 
Let $B$ be Brownian motion (started in $x$ or $y$) and write $\sigma$ for the first time at which $(t,B_t)$ leaves the rectangle. 
As a particular case of 
Hunt's switching identities, we know that
\begin{align} \label{switch}
\E^{x} \left[|B_\sigma - y|\right] = \E^y \left[ |B_\sigma - x| \right].
\end{align}
In this paper, we will be interested in generalisations of this identity, which are of particular relevance in the construction of solutions to the Skorokhod embedding problem. 
The Skorokhod embedding problem can be stated as follows: Given measures $\lambda, \mu$ on $\R$, and a Brownian motion $B$ with $B_0 \sim \lambda$, find a stopping time $\tau$ such that 
\begin{equation}\label{SEP} B_\tau \sim \mu \text{ and } (B_{t \wedge \tau})_{t \ge 0} \text{ is uniformly integrable.} \tag{$\mathsf{SEP}$} \end{equation}

In this paper we will be interested in a sub-class of solutions to \eqref{SEP} which are given by the first hitting time of a right-barrier (resp.\ left-barrier) $\mathcal R$ which is a closed subset of  $\R_+\times \R$ with the property that $(t,x) \in \mathcal{R} \implies (s,x) \in \mathcal{R}$, when $s > t$ (resp. $s<t$).
The most prominent example of such a solution is the Root solution \cite{Ro69}. It establishes the existence of a right-barrier $\mathcal R^{Root}$ such that
\[
\tau^{Root}:=\inf\{t\geq 0 : (t,B_t)\in\mathcal R^{Root}\}
\]
solves \eqref{SEP}. 
There is an analogous result for left-barriers by Rost \cite{Ro76}. While Root's original contribution was non-constructive, it was shown in \cite{CoWa13, CoWa12} how to build the Root barrier. 
This was done  by considering \emph{optimal stopping problems} and establishing the crucial identity 
\begin{equation} \label{eq:main_intro}
  - \E^\lambda \left[ |B_{\tau^{Root}\wedge T} - y| \right] = \sup_{\sigma \le T} \E^y\left[U_{\mu}(B_\sigma) \1_{\sigma < T} + U_{\lambda}(B_\sigma) \1_{\sigma = T}\right] 
\end{equation}
where 
the supremum is taken over all stopping times $\sigma \le T$. 
Here $U_m$ is the potential function associated with the measure $m$, i.e.\ \[U_m(x) := -\int |y-x| \, m(\di y).\]

From formula \eqref{eq:main_intro}	,  the barrier region $\mathcal{R}^{Root}$ can be identified with the stopping region of the optimal stopping problem \emph{after a reversal of the time component}. Together with the switching of the starting position between $y$ and $\lambda$, this formula can be seen as a generalisation of \eqref{switch}. Furthermore, the articles \cite{CoWa13, CoWa12} also provide a construction for the Rost barrier similarly based on optimal stopping problems.

The equation \eqref{eq:main_intro} is well understood in potential theoretic terms, and can be seen as a deep consequence of Markov duality, applied to the time-space process $(t,B_t)$. However this connection does not offer much insight into why such equations might arise in these applications. In \cite{CoWa13, CoWa12} this connection was made via viscosity theory and it was noted that a probabilistic explanation has yet to be given. 
\medskip\\
{\it The  aim of this short paper is to show that \eqref{eq:main_intro} is a natural consequence of some simple probabilistic arguments.} As a bonus, our proof reveals a previously unexplored symmetry between the Root and Rost solutions \cite{Ro76} to \eqref{SEP}.
\\

As a brief taste of the style of argument we will use, and to highlight its fundamental nature, let us first consider equation \eqref{switch}. Introduce a second Brownian motion $W$, independent of $B$, running from right to left and started on the right hand side of the rectangle $D$ at the point $(T,y)$. 
For any $s \in [0,T]$ we consider the stopping times
\begin{align*}
    \sigma_s &:= \sigma \wedge (T-s)
\\    \tau_s &:= \inf\{t\geq 0 : (T-t, W_t) \not\in D\} \wedge s
\end{align*}
and define 
\[
    F(s):= \E^{B_0=x, W_0=y}\left[|B_{\sigma_s} - W_{\tau_s}|\right].
\]

Then $F(0)=\E^{x}\left[|B_\sigma - y|\right]$ and $F(T)=\E^y \left[|B_\sigma - x|\right]$. We will show that the function $s\mapsto F(s)$ is constant.
In fact we will first show this for a discrete time version where we replace $B$ by a random walk $X$ and $W$ by a random walk $Y$. For this discrete version, it is not difficult to see that $F(s)=F(s-1)$ (cf. Fig.\ 1 or Section \ref{sec core 1d} for details) so that $F$ is constant. From there, \eqref{switch} follows from an application of Donsker's theorem.
\begin{figure}[H]
\begin{subfigure}{.40\linewidth}
\centering
\input{Prelude1.tex}
\end{subfigure}
\begin{subfigure}{.40\linewidth}
\centering
\input{Prelude2.tex}
\end{subfigure}
\caption{Illustration of $F(s) = F(s-1)$.}
\end{figure}
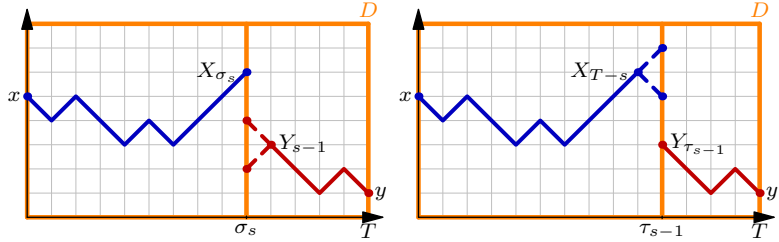

%
%
%
%
%
%
%
%
%

The Rost optimal stopping problem was subject of investigation in \cite{MC91} by McConnell where it is derived via classical PDE methods and in \cite{DA18} by De Angelis where a probabilistic proof is given relying on stochastic calculus.
Furthermore, the Root optimal stopping problem was also derived by Gassiat, Oberhauser and Zou in \cite{GaObZo19} where a suitable extension for a much wider class of Markov processes is established using classical potential theoretic methods as well as by Cox, Ob{\l}{\'o}j and Touzi in \cite{CoObTo18} where a multi-marginal extension of the problem is found.

\subsection{Overview}
In Section \ref{SSRW case} we shall establish \eqref{eq:main_intro} in the context of simple symmetric random walks (SSRW) on the integer lattice, {as well as a related identity pertaining Rost stopping times.} 
Interestingly, the symmetry between Root and Rost cases will be obtained as a consequence of a simple time-reversal principle. 
Then in Section \ref{Multidimensional case} we explore extensions to the multidimensional setting. 
In Section \ref{Brownian case} we will give some remarks on the passage to continuous time and in Section \ref{conclusions} we will draw some future perspectives.

{
\section{The Root and Rost optimal stopping problems}
}

Recall the notion of Skorokhod embedding problem \eqref{SEP} from the introduction. 
This problem was first formulated and solved by Skorokhod \cite{Sk61,Sk65}, and numerous new solutions have been found since. 
We refer to the surveys of Hobson \cite{Ho11} and Ob{\l}{\'o}j \cite{Ob04} for an account of many of these solutions.
To guarantee well-posedness, we assume thoroughout that $\lambda,\mu$ have finite first moment and are in convex order.

Let $W$ denote a one-dimensional Brownian motion (in keeping with the introduction, we think of $W$ as a Brownian motion running backwards; the reason for this will become clear in the following section).
Suppose we are given initial and target distributions $\lambda$ and $\mu$, we want to study the Root \cite{Ro69} resp.\ Rost \cite{Ro76} solution to the corresponding \eqref{SEP}.  
While Root and Rost solutions are most commonly given as hitting times of so called \emph{barriers}, specific subsets of $\R^2$, keeping the previous notation, we will denote by $D^{Root}$ (resp.\ $D^{Rost}$) the continuation set of the Root (resp.\ Rost) embeddings which can be seen as the complements of the barriers in $\R^2$.

Let us write $\mu^{Root}$  (resp. $\mu^{Rost}$) for the law of the Brownian motion starting with distribution $\lambda$ at the time it leaves $D^{Root}$ (resp. $D^{Rost}$) and $\mu_T^{Root}$  (resp.\ $\mu_T^{Rost}$) for the time it leaves $D^{Root}\cap([0,T)\times\R)$ (resp.\ $D^{Rost}\cap([0,T)\times\R)$).
Recall that the potential of a measure $m$ is denoted by 
$U_{m}(y):=-\int|y-x|\, m(\di x)$, 
and for a random variable $Z$ we write $U_Z$ for the potential of the law of $Z$. 
Throughout this note we consider \emph{optimal stopping problems}, thus suprema taken over $\tau$ (resp. $\sigma$) will denote suprema over stopping times.

The relations of interest in our article, found in \cite{CoWa13, CoWa12}, are
\begin{align} \label{toprovebrownian}
    U_{\mu_{T}^{Root}}(x)   &= \Ex \left [ U_{\mu^{Root}}\left(W_{\tau^{*} }\right)\1_{\tau^{*}<T} + U_{\lambda}\left(W_{\tau^{*}}\right)\1_{\tau^{*}=T}  \right ] \\\label{toprovesupbrownian}
                            &= \sup\limits_{ \tau\leq T}\Ex \left [ U_{\mu^{Root}}(W_{\tau})\1_{\tau<T} + U_{\lambda}(W_{\tau})\1_{\tau=T}  \right ] , 
\end{align}
where the optimizer is $\tau^*:=\inf\{t \geq 0: (T-t,W_t)\notin D^{Root} \}\wedge T$, and
\begin{align} \label{toproveRostbrownian}
 U_{\mu^{Rost}}(x)- U_{\mu_T^{Rost}}(x) &= \Ex \left[ \left(U_{\mu^{Rost}}-U_{\lambda}\right)(W_{\tau_*})  \right] 
 \\\label{toproveRostsupbrownian}
                                        &= \sup\limits_{\tau\leq T}\Ex \left[ \left(U_{\mu^{Rost}}-U_{\lambda}\right)(W_{\tau})  \right],
\end{align} 
where the optimizer is $\tau_*:=\inf\{t \geq 0: (T-t,W_t)\notin D^{Rost} \}\wedge T$.
\\

{
Remark that \eqref{toprovebrownian}-\eqref{toprovesupbrownian} are related to \eqref{eq:main_intro} in the introduction. On the other hand, \eqref{toproveRostbrownian}-\eqref{toproveRostsupbrownian} haven't been stated yet. One of the contributions of the article will be to obtain the latter as a consequence of the former, by means of a previously unexplored symmetry between Root and Rost solutions, as hinted at in the introduction.

%

In the coming section} we establish the above results \eqref{toprovebrownian}-\eqref{toproveRostsupbrownian} in the context of simple symmetric random walks (SSRW) on the integer lattice. 

%
%
%
%
%
%
%
%
%
\section{A common one-dimensional random walk framework}\label{SSRW case}
Consider a set $D \subseteq \mathbb{Z}\times \mathbb{Z}$ satisfying
\begin{itemize}
\item If $(t,m)\in D$, then for all $s<t$ also $(s,m)\in D$.
\end{itemize}

This should be seen as a discretised version of the Root continuation set defined in the introduction.
Likewise a Rost continuation set can be cast in the above form after reflection w.r.t.\ a vertical line.\\ 

\textbf{Notation:}
Denote by $X,Y$ two mutually independent SSRW on some probability space $(\Omega, \P)$ which are started at possibly random initial positions. 
Given $x,y\in \Z$ we write $\P^x$ and $\P_y$ for the conditional distribution given $X_0=x$ and $Y_0=y$ resp.\ Similarly $\P^x_y $ means that we condition on both events simultaneously.
We can consider a probability measure $\lambda$ on $\Z$ a ``starting'' distribution by setting $\P^\lambda := \sum_{x\in \Z} \P^x\lambda(\{x\})$, etc.
Let us then introduce the stopping time
\begin{align*}
\rho^{Root} &= \inf \{t\in \N:(t,X_t)\notin D\},
\end{align*}
where $\N = \{0,1,\dots\}$.
We define by $\mu^{Root}$
the law of $X_{\rho^{Root}}$ under $\P^{\lambda}$,  
and assume henceforth that the martingale $\left(X_{\rho^{Root}\wedge t}\right)_{t\in\N}$ is uniformly integrable. 
We conveniently drop the dependence of $\mu^{Root}$ on $\lambda$.
Given $T\in\N$ we write $\mu^{Root}_T$ for
the $\P^{\lambda}$-law of $X_{\rho^{Root}\wedge T}$.
Note that this definition is equivalent to $\mu^{Root}_T$ being the law of $X$ started with distribution $\lambda$ at the time it leaves $D^{Root} \cap \left(\{0,\dots, T-1\} \times \Z \right)$.
We will first establish the following identity, which is a discrete-time version of \eqref{toprovebrownian}-\eqref{toprovesupbrownian}
\begin{align} \label{toprove}
 U_{\mu_{T}^{Root}}(y)& = \Ey \left [ U_{\mu^{Root}}\left(Y_{\tau^{*} }\right)\1_{\tau^{*}<T} + U_{\lambda}\left(Y_{\tau^{*}}\right)\1_{\tau^{*}=T}  \right ] \\\label{toprovesup}
 &= \sup\limits_{\tau \leq T} \Ey \left [ U_{\mu^{Root}}(Y_{\tau})\1_{\tau<T} + U_{\lambda}(Y_{\tau})\1_{\tau=T}  \right ] ,
\end{align}
where the optimizer is 
\( \tau^*:=\inf\{t\in\N: (T-t,Y_t)\notin D \}\wedge T.\)

From here we will derive the discrete-time analogue of \eqref{toproveRostbrownian}-\eqref{toproveRostsupbrownian}.
%
%
%
%
%
\subsection{Core argument}\label{sec core 1d}

For convenience of the reader we present here the basis of the argument which we repeatedly use, namely that for $s \in \{1, \dots, T\}$
\begin{equation}\label{eq:trivial}
\Exy\left [ |X_{T-s}-Y_{s}|\right ] = \Exy\left [ |X_{T-(s-1)}-Y_{s-1}|\right ] .
\end{equation}
This represents a formalisation of the discretised argument in the
prelude. 
Indeed,
\begin{align*}
    \Exy\left [ |X_{T-s}-Y_{s}|\right ] 
        &= \Exy\left[\Exy\left[ |X_{T-s}-Y_{s}| \big| X_{T-s},Y_{s-1}\right]\right ] \\
        &= \Exy\left[\Exy\left[ |(X_{T-s} - Y_{s-1})-(Y_{s}-Y_{s-1})|\big| X_{T-s},Y_{s-1}\right]\right] \\
        &= \Exy\left[ |X_{T-s}-Y_{s-1}|+\1_{X_{T-s}=Y_{s-1}}\right ]\\
        &= \Exy\left[\Exy\left[ |(X_{T-s}-Y_{s-1})+(X_{T-s+1} -X_{T-s})|\big| X_{T-s},Y_{s-1}\right ]\right ] \\
        &= \Exy\left[\Exy\left[ |X_{T-s+1}-Y_{s-1}|\big| X_{T-s},Y_{s-1}\right ]\right ] \\
        &= \Exy\left[ |X_{T-(s-1)}-Y_{s-1}|\right ],
\end{align*}
%
which is best read from the top until the middle equality and then from the bottom until the same equality. 
More important than \eqref{eq:trivial} is the reasoning above, especially the appearance of the indicator of the event $\{X_{T-s}=Y_{s-1}\}$, which stems from the fact that $Y_{s-1}$ (resp. $X_{T-s}$) always splits into $Y_{s}\pm 1$ (resp. $X_{T-(s-1)}\pm 1$) with probability $1/2$.
For an illustration of this phenomenon also see Figure \ref{CoreArgumentPlot}.
%
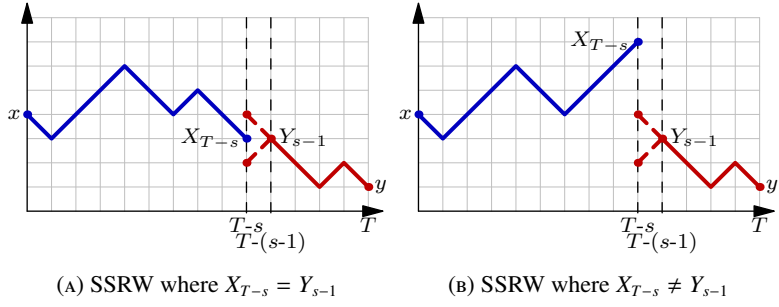
\begin{figure}
\begin{subfigure}{.40\linewidth}
\centering
\input{CoreArgument1.tex}
\caption{SSRW where $X_{T-s} = Y_{s-1}$}
\end{subfigure}
\begin{subfigure}{.40\linewidth}
\centering
\input{CoreArgument2.tex}
\caption{SSRW where $X_{T-s} \neq Y_{s-1}$}
\end{subfigure}
\caption{Illustrating the appearance of the indicator function in the core argument.}
\label{CoreArgumentPlot}
\end{figure} %
%
%
%
%
%
\subsection{The Root case}\label{rootcase}
Let $\tau^* := \min\{t \in \N: (T-t,Y_t) \not\in D\} \wedge T$. We start with a useful observation:
\begin{remark}\label{obspotentials}
The equality $U_{\mu^{Root}}\left(Y_{\tau^{*}}\right)\1_{\tau^{*}<T} =U_{\mu_T^{Root}}\left(Y_{\tau^{*}}\right)\1_{\tau^{*}<T} $ holds. 
Indeed, let $z=Y_{\tau^{*}} $ on $\{\tau^{*}<T\}$.
Then $(T-\tau^{*},z)\notin D$ and hence $(X_{T}-z)(X_{\rho^{Root}}-z)\geq 0$ on $\{\rho^{Root}>T\}$ as otherwise $X$ would have left $D$ before $\rho^{Root}$, see also Figure \ref{RootRemarkPlot}.
This implies
\begin{align}
    -U_{\mu^{Root}}(z)
        &=\Elambda\left[|X_{\rho^{Root}} - z|\right]\nonumber\\
        &=\Elambda\left[|X_{\rho^{Root}\wedge T} - z|\1_{\rho^{Root}\leq T} - (X_{\rho^{Root}} - z)\1_{\rho^{Root}> T,\, z>X_T} + (X_{\rho^{Root}} - z)\1_{\rho^{Root}> T,\, z\leq X_T}  \right] \nonumber\\
        &=\Elambda\left[|X_{\rho^{Root}\wedge T} - z|\1_{\rho^{Root}\leq T} - (X_{\rho^{Root}\wedge T} - z)\1_{\rho^{Root}> T,\, z>X_T} + (X_{\rho^{Root}\wedge T} - z)\1_{\rho^{Root}> T,\, z\leq X_T}  \right] \nonumber\\
        &=\Elambda\left[|X_{\rho^{Root}\wedge T} - z|\right]=-U_{\mu^{Root}_T}(z).
\end{align}
Accordingly we may replace $\mu^{Root}$ by $\mu^{Root}_T$ in \eqref{toprove} (but we do not do so in \eqref{toprovesup}).
\end{remark}

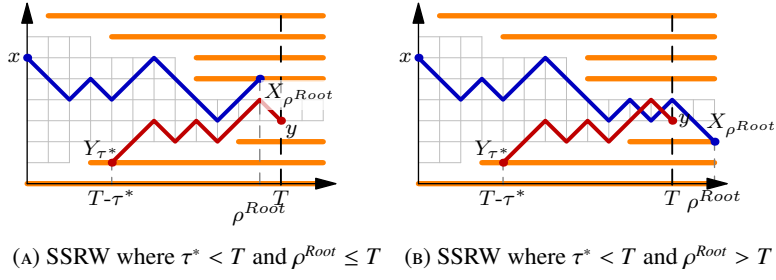
\begin{figure}
\begin{subfigure}{.40\linewidth}
\centering
\input{SSRWPlot_Remark31_1.tex}
\caption{SSRW where $\tau^* < T$ and $\rho^{Root}\leq T$}
\end{subfigure}
\begin{subfigure}{.40\linewidth}
\centering
\input{SSRWPlot_Remark31_2.tex}
\caption{SSRW where $\tau^* < T$ and $\rho^{Root} > T$}
\end{subfigure}
\caption{Illustrating properties of Root stopping times.}
\label{RootRemarkPlot}
\end{figure}

Given a $Y$-stopping time $\tau\leq T$, we define a stopping time $\sigma(\tau)$ of $X$ as the first time before $T-\tau$ that $X$ leaves $D$, i.e. $\sigma(\tau):= \rho^{Root} \wedge (T-\tau)$.\footnote{Formally, $\tau$ is a stopping time w.r.t.\ the filtration $\G=(\G_t)_{t\in\N}$,  where $\G_t=\sigma(\{Y_u:u\leq t\})$. $\sigma(\tau)$ is a stopping time w.r.t.\ the filtration $\F=(\F_s)_{s\leq T} $, where $\F_s=\sigma (\{ X_u, Y_t: u \leq s, t\leq T\})$}
For any $y \in \Z$ we now introduce the crucial interpolating function
\begin{equation}\label{interpolatingfun}
    F(s)\;\;:=\;\; F^{\tau^*}(s)\;\;:=\;\;\Elambday\left [| X_{\sigma(\tau^*\wedge s)} - Y_{\tau^*\wedge s} | \right ]\hspace{5pt}\mbox{ for }s\in\{0,\dots,T\}.
\end{equation}
It may  help to picture  $Y$ evolving ``leftwards'' from the lattice point $(T,y)$ at time zero, so that its exit time $\tau^*$ from $D$ before $T$ is measured as $T-\tau^*$ for the ``rightwards'' process $X$.
\begin{remark}\label{remark:interpolation}
We see that $\sigma(0)=\rho^{Root}\wedge T$, so consequently 
\[
    F(0)=\Elambda\left[| X_{\rho^{Root}\wedge T} - y|\right]=-U_{\mu^{Root}_{T}}(y).
\]
On the other hand,
$\sigma(\tau^*)=\rho^{Root}\wedge (T-\tau^*)$, so
\begin{align*}
F(T)    &=\Elambday\left[|X_{\rho^{Root} \wedge (T-\tau^*)} - Y_{\tau^*} |\right] 
\\      &= \Elambday\left[ |X_{\rho^{Root} \wedge (T-\tau^*)} - Y_{\tau^*}|\1_{\tau^*<T} + |X_{\rho^{Root} \wedge (T-\tau^*)} - Y_{\tau^*}|\1_{\tau^*=T} \right]
\\      &= \Elambday\left[ |X_{\rho^{Root} \wedge T} - Y_{\tau^*}|\1_{\tau^*<T} + |X_{0} - Y_{\tau^*}|\1_{\tau^*=T}\right]
\\        &= -\Ey\left[ U_{\mu^{Root}_T}\left(Y_{\tau^*}\right)\1_{\tau^*<T} + U_{\lambda}( Y_{\tau^*})\1_{\tau^*=T}\right] ,
\end{align*}
%
by independence and by applying the appropriate analogue of the argument in Remark \ref{obspotentials} for the third equality.
\end{remark}

We can now prove \eqref{toprove} and \eqref{toprovesup}; we treat the
cases separately.
%
%
%
%
%
\begin{lemma}
\label{constancydiscreteRoot}
The function $F$ is constant. Consequently
\[ U_{\mu_{T}^{Root}}(y) =\Ey \left [ U_{\mu^{Root}}\left(Y_{\tau^{*} }\right)\1_{\tau^{*}<T} + U_{\lambda}\left(Y_{\tau^{*}}\right)\1_{\tau^{*}=T}  \right ]\]
\end{lemma}
\begin{proof}
Let $0<s<T$. 
Define the stopping times $\tau^*_s:=\tau^*\wedge s$ and $\sigma_s = \sigma(\tau^* \wedge s) = \rho^{Root}\wedge(T-\tau^*_s)$.
Then
\[
    F(s) =\Elambday\left[|X_{\sigma_s} - Y_{\tau^*_s}|\right].
\]
Let us first prove that 
\begin{equation}\label{eq first to prove}
   \Elambday\left[ | X_{\sigma_s} - Y_{\tau^*_s} | \right ] 
        =\Elambday\left[ | X_{\sigma_s} - Y_{\tau^*_{s-1}} | + \1_{X_{\sigma_s} = Y_{\tau^*_{s-1}},\, \tau^*\geq s} \right ].
\end{equation}
Since
\[
   \Elambday\left[| X_{\sigma_s} - Y_{\tau^*_s}| \right ] 
        =\Elambday\left[| X_{\rho^{Root}\wedge (T-s)} - Y_{s}|\1_{{\tau^*}\geq s}  \right ] 
        +\Elambday\left[| X_{\sigma_s} - Y_{\tau^*_{s-1}}|\1_{{\tau^*}<s}  \right ],
\]
and with the appropriate analogue of the core argument \eqref{eq:trivial} (see also Figure \ref{RootCoreArgumentPlot})
\begin{align*}
    \Elambday       &\left[| X_{\rho^{Root}\wedge (T-s)} - Y_{s}|\1_{{\tau^*}\geq s}  \right ] 
\\    =&\,\,\Elambday  \left[\Elambday\left[ | X_{\rho^{Root}\wedge (T-s)} - Y_{s} | \big|X_{\rho^{Root}\wedge (T-s)}, Y_0,\dots,Y_{s-1} \right]\1_{{\tau^*}\geq s}  \right] 
\\  =&\,\,\Elambday   \left[\Elambday\left[|(X_{\rho^{Root}\wedge (T-s)}-Y_{s-1}) - ( Y_{s}-Y_{s-1})| \big| X_{\rho^{Root}\wedge (T-s)}, Y_0,\dots,Y_{s-1} \right] \1_{{\tau^*}\geq s} \right] 
\\  =&\,\,\Elambday  \left[\left( |X_{\rho^{Root}\wedge (T-s)}-Y_{s-1}|+ \1_{X_{\rho^{Root}\wedge (T-s)}=Y_{s-1}} \right)\1_{{\tau^*}\geq s} \right] 
\\  =&\,\,\Elambday   \left[| X_{\sigma_s} - Y_{\tau^*_{s-1}} |\1_{{\tau^*}\geq s}   + \1_{X_{\rho^{Root}\wedge (T-s)}=Y_{s-1} \, ,\, {\tau^*}\geq s }  \right ],
\end{align*}
%
clearly \eqref{eq first to prove} follows. 
Now let us similarly establish that
\begin{align}
   \Elambday\left[| X_{\sigma_{s-1}} - Y_{\tau^*_{s-1}}| \right ] 
        &=\Elambday\left[| X_{\sigma_s} - Y_{\tau^*_{s-1}}| + \1_{X_{\sigma_s} = Y_{\tau^*_{s-1}} ,\, \tau^*\geq s,\, \rho^{Root}> T-s} \right ]\label{eq second to prove} \\
        &=\Elambday\left[| X_{\sigma_s} - Y_{\tau^*_{s-1}}| + \1_{X_{\sigma_s} = Y_{\tau^*_{s-1}} ,\, \tau^*\geq s} \right ] .\label{eq equivalent}
\end{align}
Indeed,
\[
     \Elambday\left[| X_{\sigma_{s-1}} - Y_{\tau^*_{s-1}}| \right ] 
    =\Elambday\left[| X_{T-s+1} - Y_{{s-1}}|\1_{\tau^*\geq s,\, \rho^{Root}> T-s} \right ] 
    +\Elambday\left[| X_{\sigma_{s}} - Y_{\tau^*_{s-1}}|\1_{\{\tau^*< s\}\cup \{\rho^{Root}\leq T-s\}} \right ], 
\]
and again with the appropriate analogue of \eqref{eq:trivial}
\begin{align*}
     &\Elambday    \left[|X_{T-s+1} - Y_{{s-1}}|\1_{\tau^*\geq s,\, \rho^{Root}> T-s} \right]
\\  &\quad=\Elambday\left[\Elambday \left[|X_{T-s+1} - Y_{s-1}| \big| X_0,\dots,X_{T-s}, Y_0,\dots,Y_{s-1} \right ] \1_{{\tau^*}\geq s,\, \rho^{Root}> T-s}\right ] \\
    &\quad=\Elambday\left[\Elambday \left[| (X_{T-s}-Y_{s-1}) + (X_{T-s+1}-X_{T-s})|\big| X_0,\dots,X_{T-s}, Y_0,\dots,Y_{s-1}\right] \1_{{\tau^*}\geq s,\, \rho^{Root}> T-s} \right] 
\\  &\quad=\Elambday\left[\left(| X_{T-s} - Y_{{s-1}}| + \1_{X_{T-s}=Y_{s-1}}\right)\1_{{\tau^*}\geq s,\, \rho^{Root}> T-s}\right] \\
    &\quad=\Elambday\left[| X_{\sigma_s} - Y_{\tau^*_{s-1}}|\1_{{\tau^*}\geq s,\, \rho^{Root}> T-s}  + \1_{X_{\sigma_s}=Y_{\tau^*_{s-1}},\,{\tau^*}\geq s \, ,\, \rho^{Root}> T-s}  \right ],
\end{align*}
%
also \eqref{eq second to prove} follows.
We then see that \eqref{eq equivalent} holds true since on $\{X_{\sigma_s} = Y_{\tau^*_{s-1}},  \tau^*\geq s\}$ we have $(T-(s-1),Y_{s-1})=(T-(s-1),Y_{\tau^*_{s-1}})= (T-s+1,X_{\sigma_s})\in D$, and so by definition of $D$ necessarily $(\rho^{Root}\wedge (T-s),X_{\sigma_s})\in D$, thus $\rho^{Root}\geq T-s+1$ is fulfilled. 
The identities \eqref{eq first to prove} and \eqref{eq equivalent} now yield that $F$ is constant.
\end{proof}

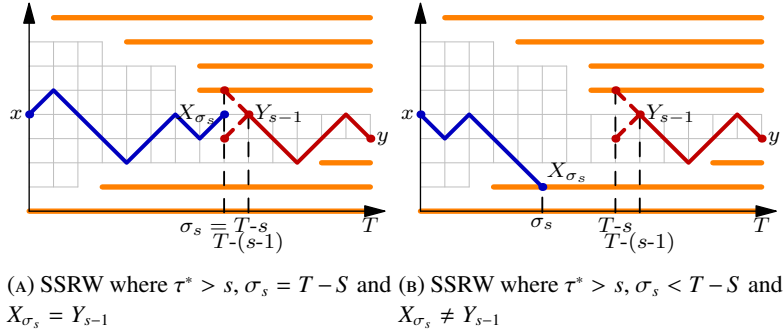
\begin{figure}
\begin{subfigure}{.40\linewidth}
\centering
\input{SSRWPlot1.tex}
\caption{SSRW where $\tau^* > s$, $\sigma_s = T-S$ and $X_{\sigma_s} = Y_{s-1}$}
\end{subfigure}
\begin{subfigure}{.40\linewidth}
\centering
\input{SSRWPlot2.tex}
\caption{SSRW where $\tau^* > s$, $\sigma_s < T-S$ and $X_{\sigma_s} \neq Y_{s-1}$}
\end{subfigure}
\caption{Illustration of the core argument in the Root setting.}
\label{RootCoreArgumentPlot}
\end{figure}
The proof of \eqref{toprovesup} follows similar lines. 
Given a $Y$-stopping time $\tau$ we consider the  interpolating function
\begin{equation}\label{interpolatingfungeneral}
F^{\tau}(s)\;\;:=\;\;\Elambday\left [| X_{\sigma(\tau\wedge s)} - Y_{\tau\wedge s} | \right ]\hspace{5pt}\mbox{ for }s\in\{0,\dots,T\}.
\end{equation}
%
\begin{lemma}
\label{monotonicitydiscreteRoot}
For every $\{0,\ldots,T\}$-valued stopping time $\tau$ of $Y$, the function $F^{\tau}$ is increasing and 
\[ U_{\mu_{T}^{Root}}(y) \geq \Ey \left [ U_{\mu^{Root}}\left(Y_{\tau }\right)\1_{\tau<T} + U_{\lambda}\left(Y_{\tau}\right)\1_{\tau=T}  \right ]\]
\end{lemma}
\begin{proof}
Clearly $F^{\tau}(0)= -U_{\mu_{T}^{Root}}(y)$. On the other hand,
\begin{align*}
    F^{\tau}(T) &=\Elambday\left[ |X_{\rho^{Root}\wedge (T-\tau)} - Y_{\tau}|\1_{\tau<T} + |X_{0} - Y_{\tau}|\1_{\tau=T} \right] \\
                &= -\Elambday\left[ U_{X_{\rho^{Root}\wedge (T-\tau)}}(Y_{\tau})\1_{\tau<T} + U_{\lambda}(Y_{\tau})\1_{\tau=T} \right] \\
                &\leq -\Elambday\left [U_{X_{\rho^{Root}}}(Y_{\tau})\1_{\tau<T} + U_{\lambda}(Y_{\tau})\1_{\tau=T}\right ] ,
\end{align*}
where the inequality is a consequence of the potentials $s\mapsto U_{X_{\rho^{Root}\wedge s}}(z)$ being decreasing in $s$ for each $z$ (by Jensen's inequality and optional sampling) and the martingale $\left(X_{\rho^{Root}\wedge t}\right)_{t \in \N}$ being uniformly integrable. 
Thus if we show that $F^{\tau}(\cdot)$ is increasing, we can conclude.

Let $0<s<T$. 
Define the stopping times $\tau_s:=\tau\wedge s$ and $\sigma_s=\rho^{Root}\wedge(T-\tau_s)$. 
Then, analogous to the proof of Lemma \ref{constancydiscreteRoot}, but replacing $\tau^*$ by $\tau$, we get
\begin{align*}
    F^{\tau}(s) &=\Elambday\left[|X_{\sigma_s} - Y_{\tau_{s-1}}| + \1_{X_{\sigma_s} = Y_{\tau_{s-1}}, \tau\geq s} \right ]
\\              &\geq\Elambday\left[| X_{\sigma_s} - Y_{\tau_{s-1}}| + \1_{X_{\sigma_s} = Y_{\tau_{s-1}}, \tau\geq s,\,\rho^{Root}\geq T-(s-1)} \right ]
                = F^{\tau}(s-1),
\end{align*}
thus $F^{\tau}(\cdot)$ is increasing.
\end{proof}
%
\begin{remark}
\label{rem generalisation}
The second part of the preceding proof, stating that the function $F^{\tau}$ is increasing, yields after trivial modifications that also
\begin{equation}\label{eq: aux}
s\mapsto F_{\sigma}^{\tau}(s):=\Elambday\left [| X_{\sigma\wedge(T- \tau\wedge s)} - Y_{\tau\wedge s} | \right ],
\end{equation}  
is increasing. 
We did not use the particular structure of $\rho^{Root}$ there.
\end{remark}

\begin{remark} There may be many other interpolating functions (which must coincide when $\tau=\tau^*$ of course). 
For example, if we replace $\sigma(\tau\wedge s)=\rho^{Root}\wedge (T - \tau \wedge s)$ by
\[ \sigma(\tau, s):=\begin{cases}
\rho^{Root} & \text{if } \tau < s\\
\rho^{Root} \wedge {(T-s)} & \text{else } 
\end{cases},
\]
and then define
\begin{equation}\label{interpolatingfungeneral_alternative}
\tilde{F}^{\tau}(s)\;\;:=\;\;\Elambday\left [| X_{\sigma(\tau, s)} - Y_{\tau\wedge s} | \right ]\hspace{5pt}\mbox{ for }s\in\{0,\dots,T\},
\end{equation}
we have $\tilde{F}^{\tau}(0)= -U_{\mu_{T}^{Root}}(y)$ and 
$\tilde{F}^{\tau}(T)=-\Elambday\left[ U_{X_{\rho^{Root}}}(Y_{\tau})\1_{\tau<T} + U_{\lambda}(Y_{\tau})\1_{\tau=T}\right]$, for each stopping time $\tau\in [0,T]$. 
This function can be seen to be increasing for each such $\tau$ and constant for $\tau^*$. \end{remark}
%
%
%
%
%
%
%
%
\subsection{The Rost case as a consequence of the Root case}\label{rostcase}
Emboldened by the results in the Root case, we could proceed to establish \eqref{toproveRostbrownian}-\eqref{toproveRostsupbrownian} in a SSRW setting via interpolating functions as well. 
It is much more illuminating and elegant, however, to deduce the Rost case from the Root one. 
We thus keep the notation as in the previous part.

\begin{proposition} \label{SymmetryProposition}
For each $x,y,T$, any stopping time $\tau$  for $Y$ such that $\Ey[|Y_{\tau}|]<\infty$, and every $\{0,\dots,T\}$-valued stopping time $\sigma$ for $X$, we have
\begin{align}\label{eq: first simmetry}
    \Ey\left[|x - Y_{\tau}|-|x - Y_{\tau\wedge T}|\right]&\leq \Exy\left[|X_{\sigma} - Y_{\tau}|-|X_{\sigma}-y| \right].
\end{align}
Suppose furthermore that 
\begin{equation}
\tau = \inf \{t\in \N: (T-t,Y_t)\notin D\},\label{eq: geom tau}
\end{equation} 
and that
$\sigma=\rho^{Root}\wedge T$. Then there is equality in \eqref{eq: first simmetry}. 
\end{proposition} 
\begin{proof}
We first prove the inequality
\begin{equation} \label{eq: intermed}
    \Ey\left[|x - Y_{\tau}|-|x - Y_{\tau\wedge T}| \right] \leq  \Exy\left[|X_{\sigma} - Y_{\tau}|-|X_{\sigma \wedge (T - \tau\wedge T)} - Y_{\tau\wedge T}|\right].
\end{equation}
This follows, on the one hand, by
\[
    \Ey\left[\left(|x - Y_{\tau}|-|x - Y_{\tau\wedge T}|\right)\1_{\tau <T} \right] = 0 \leq \Exy\left[ \left( |X_{\sigma} - Y_{\tau}|-|X_{\sigma\wedge (T-\tau)} - Y_{\tau}| \right)\1_{\tau <T} \right],
\]
where the inequality follows by Jensen's inequality and optional sampling.
Similarly we also conclude by Jensen and optional sampling that
\[
    \Ey\left[\left(|x - Y_{\tau}| - |x - Y_{\tau\wedge T}|\right)\1_{\tau \geq T} \right] \leq \Exy\left[\left(|X_{\sigma} - Y_{\tau}| - |X_{0} - Y_{ T}| \right) \1_{\tau \geq T} \right].
\]
We furthermore note that considering $F^{\tau}_{\sigma}$ as defined in
\eqref{eq: aux} for the choice $\lambda=\delta_x$ we have that the
r.h.s.\ of \eqref{eq: first simmetry},  resp.\ of \eqref{eq: intermed}, coincides with 
$\Exy \left[|X_{\sigma} - Y_{\tau}|\right] - F^{\tau}_{\sigma}(0)$, resp. $\Exy \left[|X_{\sigma} - Y_{\tau}|\right] - F^{\tau}_{\sigma}(T)$.
We can now conclude from Remark \ref{rem generalisation}, stating that $F_{\sigma}^{\tau}$ is increasing, the desired result \eqref{eq: first simmetry}. 

In the case $\sigma=\rho^{Root}\wedge T$ and $\tau$ fulfilling \eqref{eq: geom tau}, we obtain that $F_{\rho^{Root}\wedge T}^{\tau} = F^{\tau\wedge T}_{\rho^{Root}\wedge T}=F^{\tau^*}=F $ on $[0,T]$, by \eqref{eq: geom tau}, which by Lemma \ref{constancydiscreteRoot} is constant. 
So to conclude we must show that 
\[
    \Ey\left[|x - Y_{\tau}| - |x - Y_{\tau\wedge T}| \right] = \Exy\left[|X_{\rho^{Root}\wedge T} - Y_{\tau}|\right] - F(T).
\]
We can use the arguments in Remark \ref{obspotentials} resp. \ref{remark:interpolation} to obtain
\[
    \Ey\left[\left(|x - Y_{\tau}| - |x - Y_{\tau\wedge T}|\right)\1_{\tau <T} \right] = 0 =  \Exy\left[\left(|X_{\rho^{Root}\wedge T} - Y_{\tau}| - |X_{\rho^{Root} \wedge (T-\tau)} - Y_{\tau}| \right)\1_{\tau <T} \right].
\]
Similarly also
\[
    \Ey\left[\left(|x - Y_{\tau}| - |x - Y_{\tau\wedge T}|\right)\1_{\tau \geq T} \right] = \Exy\left[\left(|X_{\rho^{Root}\wedge T} - Y_{\tau}| - |X_{0} - Y_{T}| \right)\1_{\tau \geq T}\right],
\]
which concludes the proof.
\end{proof}
%
%
%
%
%
%
%
%
%
%
\begin{figure}
\begin{subfigure}{.40\linewidth}
\centering
\input{RootToRost1.tex}
\end{subfigure}
\begin{subfigure}{.40\linewidth}
\centering
\input{RootToRost2.tex}
\end{subfigure}
\caption{Illustration of the connection between $D^{Rost}$ and $D$, resp. between $\rho^{Rost}$ and $\rho^{Root}$.}
\label{RootRostPlot}
\end{figure}
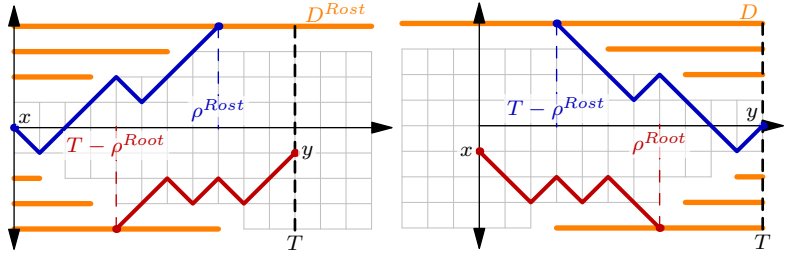

A discrete time version of the Rost optimal stopping problem \eqref{toproveRostbrownian}-\eqref{toproveRostsupbrownian} can now be established as a consequence of Proposition \ref{SymmetryProposition}.
A Rost continuation set is a set $D^{Rost}\subset \N \times \Z$ satisfying
\begin{itemize}
\item If $(t,m)\in D^{Rost}$, then for all $s>t$ also $(s,m)\in D^{Rost}$.
\end{itemize}
Given such a set for each fixed $T\in\N$ we may define $D := \{(T-t,m) : (t,m) \in D^{Rost} \}$ which is a Root continuation set to which the previous result is applicable. 
Let us introduce
 \begin{align}
\rho^{Rost} := \inf \{t\in\N:(T-t,Y_t)\notin D\} = \inf \{t\in\N:(t,Y_t)\notin D^{Rost}\},
\end{align} 
and let $\mu^{Rost}$ (resp. $\mu^{Rost}_T$) denote the law of a SSRW started with distribution $\lambda$ and stopped at time $\rho^{Rost}$ (resp. $\rho^{Rost} \wedge T$).
See Figure \ref{RootRostPlot} for an illustration of the connection between Root and Rost continuation sets and the respective hitting times.
We assume uniform integrability of $\left(Y_{\rho_{Rost} \wedge t}\right)_{t \in \N}$.

\begin{corollary}
We have
\begin{align}     
    U_{\mu^{Rost}}(x)- U_{\mu^{Rost}_T}(x)
        &= \Ex \left[\left(U_{\mu^{Rost}} - U_{\lambda}\right)(X_{\sigma_*})  \right]               \label{toproveRost}
    \\  &= \sup\limits_{\sigma \leq T}\Ex \left[\left(U_{\mu^{Rost}} - U_{\lambda}\right)(X_{\sigma})  \right],     \label{toproveRostsup}
\end{align} 
where the optimizer is given by
\[
    \sigma_*:=\rho^{Root}\wedge T = \inf \{t\in\N:(T-t,X_t)\notin D^{Rost}\}\wedge T.
\]
\end{corollary}
\begin{proof}
For $y \in \Z$ let us first consider $Y_0 = y$, i.e. $\lambda = \delta_{y}$.
Consider Proposition \ref{SymmetryProposition} for the stopping time $\tau = \rho^{Rost}$.
As
\begin{align*}
    \Ey\left[|x - Y_{\tau}|-|x - Y_{\tau\wedge T}|\right]    &= - \left( U_{\mu^{Rost}}(x) - U_{\mu^{Rost}_T}(x) \right),
\\  \Exy\left[|X_{\sigma} - Y_{\tau}|-|X_{\sigma}-y| \right]&= - \Ex \left[\left(U_{\mu^{Rost}} - U_{\lambda}\right)(X_{\sigma}) \right],
\end{align*}
due to \eqref{eq: first simmetry} we then have
\[
    U_{\mu^{Rost}}(x)- U_{\mu^{Rost}_T}(x) \leq \sup\limits_{\sigma \leq T}\Ex \left[\left(U_{\mu^{Rost}} - U_{\lambda}\right)(X_{\sigma})  \right].
\]
To prove \eqref{toproveRost} we note that $\tau = \rho^{Rost}$ satisfies \eqref{eq: geom tau}.
Thus, for $\sigma = \sigma_*$ we have equality in \eqref{eq: first simmetry} which is precisely \eqref{toproveRost} and furthermore also gives \eqref{toproveRostsup}.
As this is true for arbitrary $y \in \Z$, the extension to general $\lambda$ is clear due to identities of the form $\E^x_{\lambda}\left[|X_{\sigma} - Y_{\tau}|\right] = \sum_{y\in\Z} \Exy\left[|X_{\sigma} - Y_{\tau}|\right]  \lambda(\{y\})$.
\end{proof}
%
%
%
%
%
%
%
%
%
%

\section{The multidimensional case}\label{Multidimensional case}

We have established \eqref{toprovebrownian}-\eqref{toproveRostsupbrownian} for the integer lattice in one dimension. We shall extend this to the setting of the $d$-dimensional integer lattice $\Z^d$ for $d$ arbitrary. 

Let $Z$ be a SSRW on $\Z^d$ and let
\[z\in\Z^d \mapsto G_n(z)\, :=\,\EE^{Z_0=0}[\#\{ t\leq n:Z_t=z\}],\] 
denote the expected number of visits to site $z$ of $Z$ started in the origin, prior to $n$.
We then consider the so-called \textit{potential kernel} of the SSRW
\[z\in\Z^d \mapsto a(z)\, :=\,\lim_{n\to\infty} G_n(0)-G_n(z),\]
which is finite in any dimensions and has the desirable property that
\begin{equation}
a(z)=-\1_{z=0}+\frac{1}{2d}\sum_{z'\sim z} a(z'). \label{Laplace}
\end{equation}
Here $z'\sim z$ if $z'$ is an immediate neighbour of $z$ (corresponding to moving away from $z$ along one coordinate only, so there are $2d$ of them). 
From this follows that $\left(a(Z_t)\right)_{t\in\N}$ is a (Markovian) submartingale and by induction 
\begin{equation}
    \E\left[a(Z_{t+n})|Z_t\right]=a(Z_t)+\sum_{\ell=0}^{n-1}\P(Z_{t+\ell}=0|Z_t),\label{supermg}
\end{equation}
which is an identity we will repeatedly use. 
In the transient case ($d\geq 3$) we have that $a$ is just the negative of the expected number of visits to a point up to an additive constant. 
In the one dimensional case we have $a(\cdot)=|\cdot|$. 
We refer to \cite[Chapter 1]{Lawler_intersections} for a review of these concepts/facts.

We first observe that owing to \eqref{Laplace} the core argument \eqref{eq:trivial} in Section \ref{sec core 1d} is still valid, so for $s\in\{1,\dots, T\}$
\[
    \Exy\left[a(X_{T-s}-Y_s)\right]= \Exy\left[a(X_{T-s}-Y_{s-1})+\1_{X_{T-s}=Y_{s-1}}\right]=\Exy\left[a(X_{T-(s-1)}-Y_{s-1})\right].
\]
We shall see that all the computations we did using $z\mapsto |z|$ in the one-dimensional case are still valid for the potential kernel $a$. For a measure $\nu$ on $\Z^d$ let 
\[A.\nu(y) := - \int a(y - x)\nu(dx).\]
 As in the previous section, we denote by $X,Y$ two independent SSRW in $\Z^d$.
\begin{proposition}
Let $\lambda$ be a starting distribution in $\Z^d$ and $D^{Root}$ (resp.\ $D^{Rost}$) be Root-type (resp.\ Rost-type) continuation sets in $\Z^{d+1}$. Denote by $\mu^{Root}$ resp. $\mu_T^{Root}$ the law of a SSRW started with distribution $\lambda$ and stopped upon leaving $D^{Root}$ resp. $D^{Root} \cap \left( \{0, \dots, T-1\} \times \Z^d\right)$ (analogously for $\mu^{Rost}$ and $\mu_T^{Rost}$), and assume that the SSRW stopped when leaving $D^{Root}$ (resp.\ $D^{Rost}$) is uniformly integrable.  Then 
\begin{align} \label{G1Root}
 A.\mu_{T}^{Root}(y)& = \Ey \left [ A.\mu^{Root}\left(Y_{\tau^{*} }\right)\1_{\tau^{*}<T} + A.{\lambda}\left(Y_{\tau^{*}}\right)\1_{\tau^{*}=T}  \right ] \\\label{G2Root}
 &= \sup\limits_{ \tau\leq T}\Ey \left [ A.{\mu^{Root}}(Y_{\tau})\1_{\tau<T} + A.{\lambda}(Y_{\tau})\1_{\tau=T}  \right ] , 
\end{align}
where the optimizer is $\tau^*:=\inf\{t\in\N: (T-t,Y_t)\notin D^{Root} \}\wedge T$, and 
\begin{align}   \label{G1Rost}
 A.{\mu^{Rost}}(x)- A.{\mu_T^{Rost}}(x) &= \Ex \left[ \left(A.{\mu^{Rost}}-A.{\lambda}\right)(X_{\tau_*})  \right] 
\\              \label{G2Rost}
                                        &= \sup\limits_{\tau\leq T}\Ex \left[ \left(A.{\mu^{Rost}}-A.{\lambda}\right)(X_{\tau})  \right],
\end{align} 
where the optimizer is $\tau_*:=\inf\{t\in\N: (T-t,X_t)\notin D^{Rost} \}\wedge T$.\\
\end{proposition} 
\begin{proof}
Let us first prove \eqref{G1Root}.
In analogy to the previous section, we define an interpolating function
\begin{equation}\label{interpolatingfungeneral-G}
    F(s)\;\;:=\;\;\Elambday\left[ a( X_{\rho^{Root}\wedge (T- \tau^*\wedge s)} - Y_{\tau^*\wedge s} ) \right ]\hspace{5pt}\mbox{ for }s\in\{0,\dots,T\}.
\end{equation}
Then clearly $F(0) = - A.\mu_{T}^{Root}(y)$ and also 
\[
    F(T)=\Elambday[a(X_{\rho^{Root}\wedge(T-\tau^*)}-Y_{\tau^*}) \1_{\tau^{*}<T} ]-\Ey[A.{\lambda}\left(Y_{\tau^{*}}\right)\1_{\tau^{*}=T} ].
\]
If we establish $-\Elambday[a(X_{\rho^{Root}\wedge(T-\tau^*)}-Y_{\tau^*}) \1_{\tau^{*}<T} ]=\Ey  [ A.\mu^{Root}\left(Y_{\tau^{*} }\right)\1_{\tau^{*}<T}]$ then \eqref{G1Root} is implied by $F$ being constant.
Clearly it suffices to show that
\[
    \Elambday[a(X_{T-\tau^*}-Y_{\tau^*}) \1_{\tau^{*}<T,\, \rho^{Root}>T-\tau^*} ] =\Elambday[a(X_{\rho^{Root}}-Y_{\tau^*}) \1_{\tau^{*}<T,\, \rho^{Root}>T-\tau^*} ].
\]
Indeed 
\begin{align}\label{tower}
    \Elambday&\left[a(X_{\rho^{Root}}-Y_{\tau^*}) \1_{\tau^{*}<T,\, \rho^{Root}>T-\tau^*}\right]        \nonumber 
\\        &= \Elambday\left[\Elambday\left[a(X_{\rho^{Root}}-Y_{\tau^*}) \1_{\tau^{*}<T,\, \rho^{Root}>T-\tau^*} \big|X_0, \dots, X_T, Y_0, \dots, Y_{T-1} \right]  \right] \nonumber
\\      &= \Elambday\left[\left( a(X_{T-\tau^*}-Y_{\tau^*})+\sum_{s=T-\tau^*}^{\rho^{Root}-1} \P\left(X_{s}=Y_{\tau^*}  \big|X_0, \dots, X_T, Y_0, \dots, Y_{T-1}\right )\right)\1_{\tau^{*}<T,\, \rho^{Root}>T-\tau^*} \right] \nonumber
\\      &= \Elambday\left[ a(X_{T-\tau^*}-Y_{\tau^*})  \1_{\tau^{*}<T,\, \rho^{Root}>T-\tau^*}\right],
\end{align}
%
%
where the last line holds since, given 
$\{X_0, \dots, X_T, Y_0, \dots, Y_{T-1}\}$  on $\{\tau^{*}<T,\, \rho^{Root}>T-\tau^*\}$, 

We now prove that $F$ is indeed constant.
First we observe that
\[
    F(s)= \Elambday\left[a(X_{\rho^{Root}\wedge (T-\tau^*\wedge s)}-Y_{\tau^*\wedge s})\right] 
        =\Elambday\left[a(X_{\rho^{Root}\wedge (T-\tau^*\wedge s)}-Y_{\tau^*\wedge (s-1)}) + \1_{X_{\rho^{Root}\wedge(T-s)}= Y_{s-1}, \tau^*\geq s}\right].
\]
To see this we consider the two cases $\{\tau^* < s\}$ and $\{\tau^* \geq s\}$ separately. While the former case is clear, on the latter we apply \eqref{supermg} where we condition on $\{X_0,\dots,X_{T-s},Y_0,\dots,Y_{s-1} \}$.
Analogously but by splitting into $\{\tau^* < s\} \cup \{\rho^{Root}\leq T-s\}$ and $\{\tau^*\geq s, \rho^{Root}>T-s\}$ we obtain
\begin{align*}
   F(s-1) &= \Elambday\left[a(X_{\rho^{Root}\wedge (T-\tau^*\wedge (s-1))}-Y_{\tau^*\wedge (s-1)})\right] 
\\      &=\Elambday\left[a(X_{\rho^{Root}\wedge (T-\tau^*\wedge s)}-Y_{\tau^*\wedge (s-1)})
        +  \1_{X_{\rho^{Root}\wedge(T-s)} = Y_{s-1},\tau^*\geq s, \rho^{Root}>T-s}\right].
\end{align*}
We conclude by observing that the two appearing indicator functions are equal, since on $\{X_{\rho^{Root}\wedge(T-s)} = Y_{s-1}, \tau^*\geq s\}$ we must necessarily have $\rho^{Root}>T-s$.

To show \eqref{G2Root} define the multi dimensional equivalent of \eqref{interpolatingfungeneral}, that is for a $\{0,\dots,T\}$-valued $Y$-stopping time $\tau$ define
\begin{equation}
    F(s)\;\;:=\;\;\Elambday\left[ a( X_{\rho^{Root}\wedge (T- \tau\wedge s)} - Y_{\tau \wedge s} ) \right ]\hspace{5pt}\mbox{ for }s\in\{0,\dots,T\}.
\end{equation}
Then clearly $F^\tau(0) = -A.{\mu_T^{Root}}(y)$. Again we can use \eqref{supermg} to show that $F^{\tau}$ is increasing and furthermore
\[
    F^\tau(T) \leq - \Ey \left[ A.{\mu^{Root}}(Y_{\tau})\1_{\tau<T} + A.{\lambda}(Y_{\tau})\1_{\tau=T}  \right].
\]

The Rost case can be derived from the Root case by analogous arguments as in Section \ref{rostcase}. 
A multidimensional version of Proposition \ref{SymmetryProposition} can be proved verbatim replacing the absolute value by the function $a$ and the Jensen arguments by submartingale arguments.
The equality case follows from \eqref{supermg} exploiting the barrier structure as it was done for \eqref{tower}.
\end{proof}
%
%
%
%
%
%
%
%
%
%
\section{From the Random Walk Setting to the Continuous Case} \label{Brownian case}
While the passage to continuous time is in essence an application of Donsker-type results, we will give a more elaborate explanation using arguments established by Cox and Kinsley in \cite{CoKi19} for the one-dimensional case.
We note that all results and arguments in Section 3 are invariant under uniform scaling of the space-time grid.
Thus for each $N \in \N$ we can consider a rescaled simple symmetric random walk $Y^N$ with space step size $\frac{1}{\sqrt{N}}$ and time step size $\frac{1}{N}$ as it is done in \cite{CoKi19}. The authors discretise an \emph{optimal Skorokhod embedding problem}, an (SEP) featuring the following additional optimisation problem
\begin{equation}
	\inf_{\tau \text{ solves (SEP)}} \E[F(B_{\tau}, \tau)]. \tag{OptSEP}								
\end{equation}
It is known that for \emph{any} convex (resp. concave) function $f: \R_+ \rightarrow \R_+$, the (OptSEP) with $F(B_{\tau}, \tau) = f(\tau)$ is solved by a Root (resp. Rost) solution, see e.g. \cite{BeCoHu17}.
It is emphasised that the stopping time and the continuation set depend on the measures $\lambda$ and $\mu$ alone and not the specific choice of $f$.

Let $D$ be a Root (resp. Rost) continuation set and consider the corresponding measure $\mu = \mu^{Root}$ (resp. $\mu = \mu^{Rost}$). 
Following \cite{CoKi19} we obtain for each $N \in \N$ a discretisation $\mu^N$ of $\mu$ such that $\mu^N \rightarrow \mu$ and moreover $\lambda^N$ and $\mu^N$ are in convex order.
Similarly a discretisation $\lambda^N$ of $\lambda$ can be found such that $\lambda^N \rightarrow \lambda$. 
The authors then propose and solve a discretised version of the (OptSEP) for $\lambda^N$ and $\mu^N$.
The optimiser will again be of Root (resp. Rost) form, given as the first time a (scaled) random walk $Y^N$ leaves a Root (resp. Rost) continuation set $\hat D^N$.
Let $D^N$ denote a time-continuous and rescaled completion of the discrete continuation set $\hat D^N$.
In \cite[Chapter 5]{CoKi19} the authors then prove convergence of $D^N$ to $D$.
We note that in the more general setting considered in \cite{CoKi19} a recovery of the initial continuation set $D$ is not guaranteed. However, in the Root case this follows due to \cite{Lo70}.
An analogous uniqueness result for Rost solutions is also true, see e.g. \cite{Gr17} for a generalisation.

By convergence of the continuation sets it is easy to see that for every $T\geq 0$ we have $\mu^N_T \rightarrow \mu_T$.
As in this setting convergence of measures implies uniform convergence of potential functions, $U_{\mu^N_T} \rightarrow U_{\mu_T}$ (see \cite{Ch77} for
details), we have established convergence of the l.h.s of (3.1) to the l.h.s of (2.1).

Let $\left(W_t^{(N)}\right)_{t \geq 0}$ denote the continuous version of the rescaled random walk $Y^N$.
To avoid heavy usage of floor functions, we will assume $T \in I:=\left\{ \frac{m}{2^n}:m,n\in\N \right\}$.
If limits are then taken along the subsequence $\left(Y^{2^n}\right)_{n\in\N}$ (resp. $\left(W^{(2^n)}\right)_{n\in\N}$) there exists an $N_0 \in \N$ such that $T$ will always be a multiple of the step size $\frac{1}{2^n}$ for all $n \geq N_0$.
For arbitrary $T>0$ the results can be recovered via density arguments.
We define the following stopping times
\begin{align} \label{stptimes}
\begin{split}
    \hat\tau^{N*}   &=              \inf\{t \in \N : (NT-t, Y_t^N) \not\in \hat D^N\} \wedge NT,
\\  \bar\tau^{N*}   &=              \inf\{t >0 : (T-t, W^{(N)}_t) \not\in D^N\} \wedge T,
\\  \tau^{*}        &=              \inf\{t >0 : (T-t, W_t) \not\in D\} \wedge T,
\end{split}
\end{align}
and the functions
\begin{align*}
    G^T(x,t)      &:= U_{\mu}(x)\1_{t<T} + U_{\lambda}(x)\1_{t=T},
\\  G^T_N (x,t)   &:= U_{\mu^N}(x)\1_{t<T} + U_{\lambda^N}(x)\1_{t=T}.
\end{align*}
The rescaled results of Section 3 then read
\begin{align*}
    U_{\mu_T^N}(x) &= \E^x \left[ G^T_N\left(Y^N_{\hat\tau^{N*}}, \frac{\hat\tau^{N*}}{N} \right) \right] \tag{3.1*}
    \\                  &= \sup_{\frac{\tau}{N} \leq T} \E^x \left[G^T_N \left( Y^N_{\tau}, \frac{\tau}{N} \right)  \right]. \tag{3.2*}
\end{align*}
Or, as $\left(Y^N_{\hat\tau^{N*}}, \frac{\hat\tau^{N*}}{N} \right) = \left(W^{(N)}_{\bar\tau^{N*}}, \bar\tau^{N*} \right)$ we consider (3.1*) in $W^{(N)}$-terms
\begin{align*}
    U_{\mu_T^N}(x) &= \E^x \left[G^T_N\left(W^{(N)}_{\bar\tau^{N*}}, \bar\tau^{N*}\right) \right]. \tag{3.1**}
\end{align*}
By Lemma 5.5 and 5.6  of \cite{CoKi19} we know 
\(
\left(W^{(N)}_{\bar\tau^{N*}},\bar\tau^{N*}\right) \stackrel{P}{\rightarrow} \left(W_{\tau^{*}},\tau^{*}\right)  \text{  as } N \rightarrow \infty
\).
To see convergence of (3.1**) to (2.1) we need to show
\begin{align*}
    &\E^x \left[|G^T_N\left(W^{(N)}_{\bar\tau^{N*}},\bar\tau^{N*}\right) - G^T\left(W_{\tau^*},\tau^*\right)|\right]
\\  &\quad \leq \E^x \left[|G^T_N\left(W^{(N)}_{\bar\tau^{N*}},\bar\tau^{N*}\right) 
                    - G^T\left(W^{(N)}_{\bar\tau^{N*}},\bar\tau^{N*}\right)|\right]
                + \E^x \left[|G^T\left(W^{(N)}_{\bar\tau^{N*}},\bar\tau^{N*}\right) 
                    - G^T\left(W_{\tau^*},\tau^*\right)|\right]
    \stackrel{N \rightarrow \infty}{\longrightarrow} 0.
\end{align*}
Convergence of the first term is clear due to the fact that uniform convergence of the potential functions implies uniform convergence of $G^T_N$ to $G^T$.
Thus it remains to show convergence of the second term. 
Note that $G^T$ is usc, so it suffices to show that
\begin{equation} \label{Gconvergence}
\E^x[G^T(W_{\tau^*}, \tau^*)] \leq \liminf_{N \rightarrow \infty}\,   \E^x\left[G^T\left(W^{(N)}_{\bar\tau^{N*}},\bar\tau^{N*}\right)\right].
\end{equation}
For this, given $\e>0$ consider the auxiliary function
\[
    \tilde G^{\e}(x,t) := U_{\mu}(x)\1_{t \leq T - \e} + U_{\lambda}(x) \1_{ T - \e < t \leq T}.
\]
Then for any random variable $X$ and stopping time $\tau$ we have
\begin{equation*} 
 \E^x \left[|G^T(X,\tau) - \tilde G^{\e}(X,\tau)|\right] \leq c \cdot \P\left[ \tau \in (T-\e, T)\right].   
\end{equation*}
Combining this with the fact that $\tilde G^{\e}$ is lsc and dominating $G^T$ we get
\begin{align*}
    \E^x\left[G^T\left(W_{\tau^*}, \tau^*\right)\right] 
    &\leq \lim_{\e \searrow 0} \liminf_{N \rightarrow \infty}\,   \E^x\left[\tilde G^{\e}\left(W^{(N)}_{\bar\tau^{N*}},\bar\tau^{N*}\right)\right]
\\  &\leq \lim_{\e \searrow 0} \liminf_{N \rightarrow \infty}\, c \cdot \P\left[ \hat\tau^{N*} \in (T-\e, T)\right] +  \liminf_{N \rightarrow \infty}\, \E^x\left[G^T\left(W^{(N)}_{\bar\tau^{N*}},\bar\tau^{N*}\right)\right].
\end{align*}
Thus we are left to show that
\(
\lim_{\e \searrow 0} \liminf_{N \rightarrow \infty} \P\left[ \hat\tau^{N*} \in (T-\e, T)\right] = 0.
\)
To more easily see the arguments involving specific barrier structures, we consider the following stopping times
\begin{align*}
    \bar\rho^N  &= \inf\{t >0 : (T-t, W^{(N)}_t) \not\in D^N\} = \inf\{t >0 : (t, W^{(N)}_t) \not\in \tilde D^N\},
\\  \rho        &= \inf\{t >0 : (T-t, W_t) \not\in D\} = \inf\{t >0 : (t, W_t) \not\in \tilde D\},
\end{align*}
where $\tilde D^N$ resp. $\tilde D$ is the \emph{Rost} continuation set we obtain by reflecting $D^N$ resp. $D$ along $\left\{\frac{T}{2}\right\}\times \R$.
By \cite[Chapter 5]{CoKi19} we know that $\bar\rho^N \stackrel{P}{\rightarrow} \rho$.
Note that we have $\bar\rho^N\1_{\bar\rho^N<T} = \bar\tau^{N*}\1_{\bar\tau^{N*}<T}$.
For $0 < \tilde T \leq T$ consider
\begin{align*}
    x_-:= \sup\{y<x: (\tilde T, y) \in \tilde D\},
\\  x_+:= \inf\{y>x: (\tilde T, y) \in \tilde D\}.
\end{align*}
Since $\tilde D$ is a Rost continuation set and $\rho$ is its Brownian hitting time, we have
\[
    \P[\rho = \tilde T] = \P[W_{\tilde T} \in \{x_-,x_+\}] = 0.
\]
So, especially for any $\e > 0$ we have $\P[\rho = T - \e] = \P[\rho = T] = 0$. 
Altogether we have
\begin{align*}
    \lim_{\e \searrow 0} \liminf_{N \rightarrow \infty} \P\left[ \bar\tau^{N*} \in (T-\e, T)\right]
= \lim_{\e \searrow 0} \liminf_{N \rightarrow \infty} \P\left[ \bar\rho^{N} \in (T-\e, T)\right] 
 =  \lim_{\e \searrow 0} \P\left[ \rho \in (T-\e, T)\right] = 0,
\end{align*}
which concludes the proof of \eqref{Gconvergence}, thus the proof of convergence of (3.1**) to (2.1).
It only remains to show (2.2).
So let $\bar\tau$ be an optimiser of (2.2). 
Lemma 5.2 in \cite{CoKi19} then gives a discretisation $\tilde \sigma^N$ of $\bar\tau$ for which $Y^N_{\tilde \sigma^N}  \stackrel{a.s.}{\rightarrow} W_{\bar\tau}$ and $\frac{\tilde \sigma^N}{N} \stackrel{P}{\rightarrow} \bar\tau$.

To obtain the other inequality, for $\e \in I$ define the function
\begin{align*}
    \tilde G^{\e}_{N}(x,t)       &:= U_{\mu^N}(x)\1_{t\leq T-\e} + U_{\lambda^N}(x) \1_{T-\e<t\leq T},
\end{align*}
and by $\hat\tau^{N*}_{\e}$ resp. $\tau^{*}_{\e}$ consider the respective stopping times defined in \eqref{stptimes}, replacing $T$ by $T-\e$.
Then
\begin{align}
    &\sup_{\tau \leq T} \E^x \left[ G^T(W_{\tau}, \tau) \right]
         = \E^x \left[G^T\left(W_{\bar\tau}, \bar\tau\right) \right]               
= \lim_{\e \searrow 0} \E^x \left[\tilde G^{\e}(W_{\bar\tau},\bar\tau)\right]     \label{I+II}
\\  &\quad  \leq \lim_{\e \searrow 0} \liminf_{N \rightarrow \infty} \E^x \left[\tilde G^{\e}_N\left(Y^N_{\tilde \sigma^N}, \frac{\tilde \sigma^N}{N} \right) \right]
        \leq \lim_{\e \searrow 0} \liminf_{N \rightarrow \infty} \sup_{\frac{\tau}{N} \leq T} \E^x \left[ \tilde G^{\e}_N \left(Y^N_{\tau}, \frac{\tau}{N}\right) \right]     \label{III+IV}
\\  &\quad  \leq \lim_{\e \searrow 0} \liminf_{N \rightarrow \infty} \sup_{\frac{\tau}{N} \leq T-\e} \E^x \left[ G^{T-\e}_N \left(Y^N_{\tau}, \frac{\tau}{N}\right) \right]
        = \lim_{\e \searrow 0} \liminf_{N \rightarrow \infty} \E^x \left[ G^{T-\e}_N \left(Y^N_{\hat\tau^{N*}_{\e}}, \frac{\hat\tau^{N*}_{\e}}{N}\right) \right]      \label{V+VI}
\\  & \quad= \lim_{\e \searrow 0} \E^x \left[G^{T-\e} \left(W_{\tau^{*}_{\e}}, \tau^{*}_{\e}\right) \right]
        = \lim_{\e \searrow 0} U_{\mu_{T-\e}}(x)
        = U_{\mu_{T}}(x)
        = \E^x \left[ G^T \left(W_{\tau^*}, \tau^*\right) \right].    \label{VII+VIII+IX+X}
\end{align}
The fact that $\lim_{\e \searrow 0}  \P\left[ \bar\tau \in (T-\e, T)\right] = 0$ gives \eqref{I+II} and that $\tilde G^{\e}$ is l.s.c gives \eqref{III+IV}.
To see \eqref{V+VI} consider the function
\[
    H^{\e}_N(x,t):= U_{\mu^N}(x)\1_{t < T-\e} + U_{\lambda^N}(x) \1_{T-\e\leq t\leq T}.
\]
Then $H^{\e}_N(x,t) \geq G^{\e}_N(x,t)$ for all $(x,t) \in \R \times [0,T]$ and trivially
\[
    \sup_{\frac{\tau}{N} \leq T} \E^x \left[ G^{\e}_N\left(Y^N_{\tau}, \frac{\tau}{N}\right) \right] 
        \leq \sup_{\frac{\tau}{N} \leq T} \E^x \left[H^{\e}_N\left(Y^N_{\tau}, \frac{\tau}{N}\right) \right].
\]
Let $(Z)_{t \geq 0}$ be a martingale, then  $\left(H^{\e}_N\left(Z_{t}, t\right)\right)_{t \in [T-\e,T]} = \left(U_{\lambda^N}\left(Z_t\right)\right)_{t \in [T-\e,T]}$ is a supermartingale as $U_{\lambda^N}$ is a concave function. 
So for any stopping time $\tau$ we have
\begin{equation} \label{supmartingale}
    \E^x \left[ H^{\e}_N \left(Z_{\tau \wedge (T-\e)},\tau \wedge (T-\e) \right) \right] \geq \E^x \left[ H^{\e}_N \left(Z_{\tau \wedge T},\tau \wedge T \right) \right].
\end{equation}
We see that no optimiser of $\sup_{\frac{\tau}{N} \leq T} \E^x \left[H^{\e}_N\left(Y^N_{\tau}, \frac{\tau}{N}\right) \right]$ will stop after time $T-\e$, as this would decrease the value of the objective function.
So we have
\begin{align*}
    \sup_{\frac{\tau}{N} \leq T} \E^x \left[H^{\e}_N\left(Y^N_{\tau}, \frac{\tau}{N}\right) \right]
        = \sup_{\frac{\tau}{N} \leq T-\e} \E^x \left[H^{\e}_N\left(Y^N_{\tau}, \frac{\tau}{N}\right) \right]
        = \sup_{\frac{\tau}{N} \leq T-\e} \E^x \left[G^{T-\e}_N\left(Y^N_{\tau}, \frac{\tau}{N}\right) \right].
\end{align*}
As we know that $\hat\tau^{N*}_{\e}$ is the optimiser of this optimal stopping problem, \eqref{V+VI} follows.
Lastly, \eqref{VII+VIII+IX+X} is due to the convergence result of (3.1*) to (2.1).

To prove convergence of the Rost optimal stopping problem replace the functions $G^T$ and $G^T_N$ above by the following functions
\begin{align*}
    G^T(x,t) &= G(x)     := U_{\mu}(x) - U_{\lambda}(x),
\\  G^T_N (x,t) &= G_N(x)   := U_{\mu^N}(x) - U_{\lambda^N}(x).
\end{align*}
We can now derive our convergence results analogous to the Root case.
%
%
%
%
%
%
%
%
%
%
\section{Perspectives} \label{conclusions} 
We illustrated the elusive connection between Root and Rost's solutions to the (SEP) and optimal stopping problems. 
Specialising to the simplest possible setting, this note restricts itself to the case of SSRW and Brownian motion.  
In a recent article by Gassiat et.~al.~\cite{GaObZo19} the analytic connection between Root solutions to the (SEP) and solutions to optimal stopping problems was established for a much more general class of Markov processes.  
This suggests that our probabilistic arguments would also hold in this generalised setting.
The extension to more general martingales should follow via analogous arguments to the extension made in Chapter~\ref{Multidimensional case} by using the appropriate potential kernel, however for non-martingales some arguments need to be replaced.
%

\bibliographystyle{plain}
\bibliography{bibSwitchingIdentities}

\end{document}

%% file: Prelude1.tex
\begin{asy}[width=1\textwidth]
import graph;
import stats;
import patterns;

defaultpen(fontsize(8pt));

// define boundary points
int T = 14;
int xmax = 8;
int t = 5;
int s = T-t;

// resp. starting points
int x = 5;
int y = 1;

// grid 
pen grid = mediumgray + 0.2;

for(int i=1; i <= xmax; ++i){
draw((0,i)--(T,i), grid);
}

for(int i=1; i <= T; ++i){
draw((i,0)--(i,xmax), grid);
} 

draw((0,0)--(0,xmax), orange + 1.7);
draw((0,xmax)--(T, xmax), orange + 1.7);
draw((T, xmax)--(T,0), orange + 1.7);
draw((T,0)--(0,0), orange + 1.7);
draw((s, xmax)--(s,0), orange + 1.7);

// horizontal lines
pen ph = black + 0.5 + dashed;

//draw((s,0-0.07)--(s,xmax), ph);
//draw((s+1,0-0.07)--(s+1,xmax), ph);

// X-path
real X0 = x;
real[] X; // random walk
X[0] = X0;
path RWX;
RWX = (0,X[0]);

// Get a path:
/* Random Path Generation
real[] pathX; // path 
srand(15); //seed for X
for(int i=1; i <= T-t; ++i){
    real coin = Gaussrand();
    if(coin > 0){
        pathX[i] = 1 ;
    }
    else{
        pathX[i] = -1;
    }
}
*/ 

// Get path manually:
real[] pathX = {-1,1,-1,-1,1,-1,1,1,1};

for(int i=1; i <= T-t; ++i){
    X[i] = X[i-1] + pathX[i-1];
    RWX = RWX--(i,X[i]);
}

// Y-path
real Y0 = y;
real[] Y; // random walk
Y[0] = Y0;
path RWY;
RWY = (T,Y[0]);

// Get a Y-path:
/* Random Path Generation
real[] pathY; // path 
srand(9); //seed for Y
for(int i=1; i <= t-1; ++i){
    real coin = Gaussrand();
    if(coin > 0){
        pathY[i] = 1 ;
    }
    else{
        pathY[i] = -1;
    }
}
*/ 

// Get path manually:
// real[] pathY = {1,1,-1,-1};
real[] pathY = {1,-1,1,1};

for(int i=1; i <= t-1; ++i){
    Y[i] = Y[i-1] + pathY[i-1];
    RWY = RWY--(T-i,Y[i]);
}

/*
srand(9); //seed for Y
for(int i=1; i <= t-1; ++i){
    real coin = Gaussrand();
    if(coin > 0){
        Y[i] = Y[i-1] + 1;
    }
    else{
        Y[i] = Y[i-1] - 1;
    }
    RWY = RWY--(T-i,Y[i]);
}
*/

// pens for graphs
pen pX = heavyblue + 1.5;
pen pXX = heavyblue + 1.5 + linetype(new real[] {4,3});

pen pY = heavyred + 1.5;
pen pYY = heavyred + 1.5 + linetype(new real[] {4,3});

// draw random paths
draw(RWX,pX);
draw(RWY,pY);

// draw extensions

draw((s+1,Y[t-1])--(s,Y[t-1]+1),pYY);
draw((s+1,Y[t-1])--(s,Y[t-1]-1),pYY);

// AXIS
pen axis = black + 0.5;
draw((0,0)--(0,xmax+0.6),axis,Arrow);
draw((0,0)--((T+0.6),0),axis,Arrow);

// plot starting dots;
dot((0,x),heavyblue);
dot((T,y),heavyred);

// plot end dots;
dot((s,X[s]),heavyblue);
dot((s+1,Y[t-1]),heavyred);

dot((s,Y[t-1]+1),heavyred);
dot((s,Y[t-1]-1),heavyred);

// TIME T TICK
draw((T,-0.05)--(T,0.05), axis);
// s TICK
draw((s,-0.05)--(s,0.05), axis);

// LABELS
label("$T$",(T,0),S);
label("$\sigma_{s}$",(s,0),S);
//label("$s$",(s,0),S);
//label("$s+1$",(s+1,-1),S);

label("$x$",(0,x),W);
label("$y$",(T,y),E);

label("$X_{\sigma_s}$",(s-0.07,X[s]),W);
label("$Y_{s-1}$",(s+1+0.06,Y[t-1]),E);

label("$D$", (T, xmax),N, orange);

// BRACE
// pen brc = black + 0.7;
// draw(brace((T,X[T-t]), (T, Y[t-1])), brc);
// label("$X_{T-t} - Y_{t-1}$", (T+1,(X[T-t]+Y[t-1])/2),E);
\end{asy}

%% file: Prelude2.tex
\begin{asy}[width=1\textwidth]
import graph;
import stats;
import patterns;

defaultpen(fontsize(8pt));

// define boundary points
int T = 14;
int xmax = 8;
int t = 5;
int s = T-t;

// resp. starting points
int x = 5;
int y = 1;

// grid 
pen grid = mediumgray + 0.2;

for(int i=1; i <= xmax; ++i){
draw((0,i)--(T,i), grid);
}

for(int i=1; i <= T; ++i){
draw((i,0)--(i,xmax), grid);
} 

draw((0,0)--(0,xmax), orange + 1.7);
draw((0,xmax)--(T, xmax), orange + 1.7);
draw((T, xmax)--(T,0), orange + 1.7);
draw((T,0)--(0,0), orange + 1.7);
draw((s+1, xmax)--(s+1,0), orange + 1.7);

// horizontal lines
pen ph = black + 0.5 + dashed;

//draw((s,0-0.07)--(s,xmax), ph);
//draw((s+1,0-0.07)--(s+1,xmax), ph);

// X-path
real X0 = x;
real[] X; // random walk
X[0] = X0;
path RWX;
RWX = (0,X[0]);

// Get a path:
/* Random Path Generation
real[] pathX; // path 
srand(15); //seed for X
for(int i=1; i <= T-t; ++i){
    real coin = Gaussrand();
    if(coin > 0){
        pathX[i] = 1 ;
    }
    else{
        pathX[i] = -1;
    }
}
*/ 

// Get path manually:
real[] pathX = {-1,1,-1,-1,1,-1,1,1,1};

for(int i=1; i <= T-t; ++i){
    X[i] = X[i-1] + pathX[i-1];
    RWX = RWX--(i,X[i]);
}

// Y-path
real Y0 = y;
real[] Y; // random walk
Y[0] = Y0;
path RWY;
RWY = (T,Y[0]);

// Get a Y-path:
/* Random Path Generation
real[] pathY; // path 
srand(9); //seed for Y
for(int i=1; i <= t-1; ++i){
    real coin = Gaussrand();
    if(coin > 0){
        pathY[i] = 1 ;
    }
    else{
        pathY[i] = -1;
    }
}
*/ 

// Get path manually:
// real[] pathY = {1,1,-1,-1};
real[] pathY = {1,-1,1,1};

for(int i=1; i <= t-1; ++i){
    Y[i] = Y[i-1] + pathY[i-1];
    RWY = RWY--(T-i,Y[i]);
}

/*
srand(9); //seed for Y
for(int i=1; i <= t-1; ++i){
    real coin = Gaussrand();
    if(coin > 0){
        Y[i] = Y[i-1] + 1;
    }
    else{
        Y[i] = Y[i-1] - 1;
    }
    RWY = RWY--(T-i,Y[i]);
}
*/

// pens for graphs
pen pX = heavyblue + 1.5;
pen pXX = heavyblue + 1.5 + linetype(new real[] {4,3});

pen pY = heavyred + 1.5;
pen pYY = heavyred + 1.5 + linetype(new real[] {4,3});

// draw random paths
draw(RWX,pX);
draw(RWY,pY);

draw((s,X[s])--(s+1,X[s]+1),pXX);
draw((s,X[s])--(s+1,X[s]-1),pXX);

//draw((T-t+1,Y[t-1])--(T-t,Y[t-1]+1),pYY);
//draw((T-t+1,Y[t-1])--(T-t,Y[t-1]-1),pYY);

// AXIS
pen axis = black + 0.5;
draw((0,0)--(0,xmax+0.6),axis,Arrow);
draw((0,0)--((T+0.6),0),axis,Arrow);

// plot starting dots;
dot((0,x),heavyblue);
dot((T,y),heavyred);

// plot end dots;
dot((s,X[s]),heavyblue);
dot((s+1,X[s]+1),heavyblue);
dot((s+1,X[s]-1),heavyblue);
dot((s+1,Y[t-1]),heavyred);

// TIME T TICK
draw((T,-0.05)--(T,0.05), axis);
// s TICK
draw((s+1,-0.05)--(s+1,0.05), axis);

// LABELS
label("$T$",(T,0),S);
//label("$s$",(s,0),S);
//label("$\sigma_{s} = T - \tau_{s}$",(s,0),S);
label("$\tau_{s-1}$",(s+1,0),S);

label("$x$",(0,x),W);
label("$y$",(T,y),E);

label("$X_{T-s}$",(s-0.07,X[T-t]),W);
label("$Y_{\tau_{s-1}}$",(s+1+0.06,Y[t-1]),E);

label("$D$", (T, xmax),N, orange);

\end{asy}

%% file: CoreArgument1.tex
\begin{asy}[width=1\textwidth]
import graph;
import stats;
import patterns;

defaultpen(fontsize(8pt));

// define boundary points
int T = 14;
int xmax = 8;
int t = 5;

// resp. starting points
int x = 4;
int y = 1;

// grid 
pen grid = mediumgray + 0.2;

for(int i=1; i <= xmax; ++i){
draw((0,i)--(T,i), grid);
}

for(int i=1; i <= T; ++i){
draw((i,0)--(i,xmax), grid);
} 

// horizontal lines
pen ph = black + 0.5 + dashed;

draw((T-t,-0.07)--(T-t,xmax), ph);
draw((T-t+1,-0.07)--(T-t+1,xmax), ph);

// X-path
real X0 = x;
real[] X; // random walk
X[0] = X0;
path RWX;
RWX = (0,X[0]);

// Get a path:
/* Random Path Generation
real[] pathX; // path 
srand(15); //seed for X
for(int i=1; i <= T-t; ++i){
    real coin = Gaussrand();
    if(coin > 0){
        pathX[i] = 1 ;
    }
    else{
        pathX[i] = -1;
    }
}
*/ 

// Get path manually:
real[] pathX = {-1,1,1,1,-1,-1,1,-1,-1};

for(int i=1; i <= T-t; ++i){
    X[i] = X[i-1] + pathX[i-1];
    RWX = RWX--(i,X[i]);
}

// Y-path
real Y0 = y;
real[] Y; // random walk
Y[0] = Y0;
path RWY;
RWY = (T,Y[0]);

// Get a Y-path:
/* Random Path Generation
real[] pathY; // path 
srand(9); //seed for Y
for(int i=1; i <= t-1; ++i){
    real coin = Gaussrand();
    if(coin > 0){
        pathY[i] = 1 ;
    }
    else{
        pathY[i] = -1;
    }
}
*/ 

// Get path manually:
real[] pathY = {1,-1,1,1};

for(int i=1; i <= t-1; ++i){
    Y[i] = Y[i-1] + pathY[i-1];
    RWY = RWY--(T-i,Y[i]);
}

/*
srand(9); //seed for Y
for(int i=1; i <= t-1; ++i){
    real coin = Gaussrand();
    if(coin > 0){
        Y[i] = Y[i-1] + 1;
    }
    else{
        Y[i] = Y[i-1] - 1;
    }
    RWY = RWY--(T-i,Y[i]);
}
*/

// pens for graphs
pen pX = heavyblue + 1.5;
pen pY = heavyred + 1.5;
pen pYY = heavyred + 1.5 + linetype(new real[] {4,3});

// draw random paths
draw(RWX,pX);
draw(RWY,pY);

draw((T-t+1,Y[t-1])--(T-t,Y[t-1]+1),pYY);
draw((T-t+1,Y[t-1])--(T-t,Y[t-1]-1),pYY);

// AXIS
pen axis = black + 0.5;
draw((0,0)--(0,xmax+0.6),axis,Arrow);
draw((0,0)--((T+0.6),0),axis,Arrow);

// plot starting dots;
dot((0,x),heavyblue);
dot((T,y),heavyred);

// plot end dots;
dot((T-t,X[T-t]),heavyblue);
dot((T-t+1,Y[t-1]),heavyred);
dot((T-t,Y[t-1]+1),heavyred);
dot((T-t,Y[t-1]-1),heavyred);

// TIME T TICK
draw((T,-0.05)--(T,0.05), axis);

// LABELS
label("$T$",(T,0),S);
label("$T$-$s$",(T-t,0),S);
label("$T$-$(s$-$1)$",(T-t+1,-0.6),S);

label("$x$",(0,x),W);
label("$y$",(T,y),E);

label("$X_{T-s}$",(T-t,X[T-t]),W);
label("$Y_{s-1}$",(T-t+1+0.06,Y[t-1]),E);
\end{asy}

%% file: CoreArgument2.tex
\begin{asy}[width=1\textwidth]
import graph;
import stats;
import patterns;

defaultpen(fontsize(8pt));

// define boundary points
int T = 14;
int xmax = 8;
int t = 5;

// resp. starting points
int x = 4;
int y = 1;

// grid 
pen grid = mediumgray + 0.2;

for(int i=1; i <= xmax; ++i){
draw((0,i)--(T,i), grid);
}

for(int i=1; i <= T; ++i){
draw((i,0)--(i,xmax), grid);
} 

// horizontal lines
pen ph = black + 0.5 + dashed;

draw((T-t,-0.07)--(T-t,xmax), ph);
draw((T-t+1,-0.07)--(T-t+1,xmax), ph);

// X-path
real X0 = x;
real[] X; // random walk
X[0] = X0;
path RWX;
RWX = (0,X[0]);

// Get a path:
/* Random Path Generation
real[] pathX; // path 
srand(15); //seed for X
for(int i=1; i <= T-t; ++i){
    real coin = Gaussrand();
    if(coin > 0){
        pathX[i] = 1 ;
    }
    else{
        pathX[i] = -1;
    }
}
*/ 

// Get path manually:
real[] pathX = {-1,1,1,1,-1,-1,1,1,1};

for(int i=1; i <= T-t; ++i){
    X[i] = X[i-1] + pathX[i-1];
    RWX = RWX--(i,X[i]);
}

// Y-path
real Y0 = y;
real[] Y; // random walk
Y[0] = Y0;
path RWY;
RWY = (T,Y[0]);

// Get a Y-path:
/* Random Path Generation
real[] pathY; // path 
srand(9); //seed for Y
for(int i=1; i <= t-1; ++i){
    real coin = Gaussrand();
    if(coin > 0){
        pathY[i] = 1 ;
    }
    else{
        pathY[i] = -1;
    }
}
*/ 

// Get path manually:
// real[] pathY = {1,1,-1,-1};
real[] pathY = {1,-1,1,1};

for(int i=1; i <= t-1; ++i){
    Y[i] = Y[i-1] + pathY[i-1];
    RWY = RWY--(T-i,Y[i]);
}

/*
srand(9); //seed for Y
for(int i=1; i <= t-1; ++i){
    real coin = Gaussrand();
    if(coin > 0){
        Y[i] = Y[i-1] + 1;
    }
    else{
        Y[i] = Y[i-1] - 1;
    }
    RWY = RWY--(T-i,Y[i]);
}
*/

// pens for graphs
pen pX = heavyblue + 1.5;
pen pY = heavyred + 1.5;
pen pYY = heavyred + 1.5 + linetype(new real[] {4,3});

// draw random paths
draw(RWX,pX);
draw(RWY,pY);

draw((T-t+1,Y[t-1])--(T-t,Y[t-1]+1),pYY);
draw((T-t+1,Y[t-1])--(T-t,Y[t-1]-1),pYY);

// AXIS
pen axis = black + 0.5;
draw((0,0)--(0,xmax+0.6),axis,Arrow);
draw((0,0)--((T+0.6),0),axis,Arrow);

// plot starting dots;
dot((0,x),heavyblue);
dot((T,y),heavyred);

// plot end dots;
dot((T-t,X[T-t]),heavyblue);
dot((T-t+1,Y[t-1]),heavyred);
dot((T-t,Y[t-1]+1),heavyred);
dot((T-t,Y[t-1]-1),heavyred);

// TIME T TICK
draw((T,-0.05)--(T,0.05), axis);

// LABELS
label("$T$",(T,0),S);
label("$T$-$s$",(T-t,0),S);
label("$T$-$(s$-$1)$",(T-t+1,-0.6),S);

label("$x$",(0,x),W);
label("$y$",(T,y),E);

label("$X_{T-s}$",(T-t-0.07,X[T-t]),W);
label("$Y_{s-1}$",(T-t+1+0.06,Y[t-1]),E);
\end{asy}

%% file: SSRWPlot_Remark31_1.tex
\begin{asy}[width=1\textwidth]

import graph;
import stats;
import patterns;

defaultpen(fontsize(8pt));

// define boundary points
int Tmax = 14;
int T = 12;
int xmax = 8;
int t = 4;
int hitX = 11; //14 for 2nd plot
int hitY = 8;
real tick = 0.05;

/*
int T = 14;
int xmax = 8;
int t = 3;
int hitX = T-t;
*/

// resp. starting points
int x = 6;
int y = 3;

// grid 
pen grid = mediumgray + 0.2;

// MANUAL GRID
// horizontal lines

draw((0,1)--(2,1), grid);
draw((0,2)--(9,2), grid);
draw((0,3)--(Tmax,3), grid);
draw((0,4)--(Tmax,4), grid);
draw((0,5)--(7,5), grid);
draw((0,6)--(7,6), grid);
draw((0,7)--(3,7), grid);

//vertical lines

draw((1,1)--(1,7), grid);
draw((2,1)--(2,7), grid);
draw((3,2)--(3,7), grid);
draw((4,2)--(4,6), grid);
draw((5,2)--(5,6), grid);
draw((6,2)--(6,6), grid);
draw((7,2)--(7,6), grid);
draw((8,2)--(8,4), grid);
draw((9,2)--(9,4), grid);
draw((10,3)--(10,4), grid);
draw((11,3)--(11,4), grid);
draw((12,3)--(12,4), grid);
draw((13,3)--(13,4), grid);
draw((14,3)--(14,4), grid);

// DRAW A BARRIER
pen barrier = orange + 2;

// horizontal lines
draw((1,8)--(Tmax,8),barrier);
draw((4,7)--(Tmax,7),barrier);
draw((8,6)--(Tmax,6),barrier);
draw((8,5)--(Tmax,5),barrier);
draw((10,2)--(Tmax,2),barrier);
draw((3,1)--(Tmax,1),barrier);
draw((0,0)--(Tmax,0),barrier);

// horizontal lines
pen ph = black + 0.7 + dashed;
pen ph_grey = grey + 0.5 + dashed;

draw((T,0 - 0.05)--(T,xmax), ph);
//draw((Tmax,0 - 0.05)--(Tmax,xmax), ph_grey);

draw((hitX,0 - 0.05)--(hitX,5), ph_grey);
draw((T-hitY,0 - 0.05)--(T-hitY,1), ph_grey);

//draw((hitX,0 - 0.05)--(hitX,1), ph);
//draw((T-t,0)--(T-t,5), heavyred + 0.5 + dashed);

// X-path
real X0 = x;
real[] X; // random walk
X[0] = X0;
path RWX;
RWX = (0,X[0]);

// Get path manually:
real[] pathX = {-1,-1,1,-1,1,1,-1,-1, -1,1,1};

for(int i=1; i <= hitX; ++i){
    X[i] = X[i-1] + pathX[i-1];
    RWX = RWX--(i,X[i]);
}

// Y-path
real Y0 = y;
real[] Y; // random walk
Y[0] = Y0;
path RWY;
RWY = (T,Y[0]);

// Get path manually:
real[] pathY = {1,-1,-1,1,-1,1, -1, -1};

for(int i=1; i <= hitY; ++i){
    Y[i] = Y[i-1] + pathY[i-1];
    RWY = RWY--(T-i,Y[i]);
}

// pens for graphs
pen pX = heavyblue + 1.5;
pen pY = heavyred + 1.5;

    // Draw paths
draw(RWX,pX);
draw(RWY,pY);

// AXIS
pen axis = black + 0.5;
draw((0,0)--(0,xmax+0.6),axis,Arrow);
draw((0,0)--((Tmax+0.6),0),axis,Arrow);

// plot starting dots;
dot((0,x),heavyblue);
dot((T,y),heavyred);

// plot end dots;
dot((hitX,X[hitX]),heavyblue);
dot((T-hitY,Y[hitY]),heavyred);

// TIME T TICK
draw((T,-0.05)--(T,0.05), axis);

// LABELS
label("$T$",(T,0),S);
//label("$T$-$s$",(T-t,0),S);
label("$T$-$\tau^*$",(T-hitY,0),S);
label("$\rho^{Root}$",(hitX,-0.6),S);

//label("$\sigma_s = T$-$s$",(hitX,0),S);

label("$x$",(0,x),W);
label("$y$",(T,y),SE);

//label("$X_{\rho^{Root}}$",(hitX,X[hitX]),SE,Fill(white+opacity(0.7)));
label("$X_{\rho^{Root}}$",(hitX,X[hitX]),SE,Fill(white+opacity(0.7)));
//label("$Y_{\tau^*}$",(T-hitY+0.6,Y[hitY]),NW,Fill(white+opacity(0.7)));
label("$Y_{\tau^*}$",(T-hitY+0.6,Y[hitY]),NW);

label("\phantom{$X_{\rho^{Root}}$}",(14-0.6,2),NE);

// BRACE
// pen brc = black + 0.7;
// draw(brace((T,X[T-t]), (T, Y[t-1])), brc);
// label("$X_{T-t} - Y_{t-1}$", (T+1,(X[T-t]+Y[t-1])/2),E);
\end{asy}

%% file: SSRWPlot_Remark31_2.tex
\begin{asy}[width=1\textwidth]

import graph;
import stats;
import patterns;

defaultpen(fontsize(8pt));

// define boundary points
int Tmax = 14;
int T = 12;
int xmax = 8;
int t = 4;
int hitX = 14;
int hitY = 8;
real tick = 0.05;

// resp. starting points
int x = 6;
int y = 3;

// grid 
pen grid = mediumgray + 0.2;

// MANUAL GRID
// horizontal lines

draw((0,1)--(2,1), grid);
draw((0,2)--(9,2), grid);
draw((0,3)--(Tmax,3), grid);
draw((0,4)--(Tmax,4), grid);
draw((0,5)--(7,5), grid);
draw((0,6)--(7,6), grid);
draw((0,7)--(3,7), grid);

//vertical lines

draw((1,1)--(1,7), grid);
draw((2,1)--(2,7), grid);
draw((3,2)--(3,7), grid);
draw((4,2)--(4,6), grid);
draw((5,2)--(5,6), grid);
draw((6,2)--(6,6), grid);
draw((7,2)--(7,6), grid);
draw((8,2)--(8,4), grid);
draw((9,2)--(9,4), grid);
draw((10,3)--(10,4), grid);
draw((11,3)--(11,4), grid);
draw((12,3)--(12,4), grid);
draw((13,3)--(13,4), grid);
draw((14,3)--(14,4), grid);

// DRAW A BARRIER
pen barrier = orange + 2;

// horizontal lines
draw((1,8)--(Tmax,8),barrier);
draw((4,7)--(Tmax,7),barrier);
draw((8,6)--(Tmax,6),barrier);
draw((8,5)--(Tmax,5),barrier);
draw((10,2)--(Tmax,2),barrier);
draw((3,1)--(Tmax,1),barrier);
draw((0,0)--(Tmax,0),barrier);

// horizontal lines
pen ph = black + 0.7 + dashed;
pen ph_grey = grey + 0.5 + dashed;

draw((T,0 - 0.05)--(T,xmax), ph);
draw((hitX,0 - 0.05)--(hitX,2), ph_grey);
draw((T-hitY,0 - 0.05)--(T-hitY,1), ph_grey);

// X-path
real X0 = x;
real[] X; // random walk
X[0] = X0;
path RWX;
RWX = (0,X[0]);

// Get path manually:
real[] pathX = {-1,-1,1,-1,1,1,-1,-1, -1,1,-1, 1, -1, -1};

for(int i=1; i <= hitX; ++i){
    X[i] = X[i-1] + pathX[i-1];
    RWX = RWX--(i,X[i]);
}

// Y-path
real Y0 = y;
real[] Y; // random walk
Y[0] = Y0;
path RWY;
RWY = (T,Y[0]);

// Get path manually:
real[] pathY = {1,-1,-1,1,-1,1, -1, -1};

for(int i=1; i <= hitY; ++i){
    Y[i] = Y[i-1] + pathY[i-1];
    RWY = RWY--(T-i,Y[i]);
}

// pens for graphs
pen pX = heavyblue + 1.5;
pen pY = heavyred + 1.5;

    // Draw paths
draw(RWX,pX);
draw(RWY,pY);

// AXIS
pen axis = black + 0.5;
draw((0,0)--(0,xmax+0.6),axis,Arrow);
draw((0,0)--((Tmax+0.6),0),axis,Arrow);

// plot starting dots;
dot((0,x),heavyblue);
dot((T,y),heavyred);

// plot end dots;
dot((hitX,X[hitX]),heavyblue);
dot((T-hitY,Y[hitY]),heavyred);

// TIME T TICK
draw((T,-0.05)--(T,0.05), axis);

// LABELS
//label("$T= \rho^{Root}\wedge T$",(T,0),S);
label("$T$",(T,0),S);
//label("$T$-$s$",(T-t,0),S);
label("$T$-$\tau^*$",(T-hitY,0),S);
label("$\rho^{Root}$",(hitX,0),S);
label("\phantom{$\rho^{Root}$}",(11,-0.6),S);

//label("$\sigma_s = T$-$s$",(hitX,0),S);

label("$x$",(0,x),W);
label("$y$",(T,y),E);

//label("$X_{\rho^{Root}}$",(hitX-0.6,X[hitX]),NE,Fill(white+opacity(0.7)));
label("$X_{\rho^{Root}}$",(hitX-0.6,X[hitX]),NE);
//label("$Y_{\tau^*}$",(T-hitY+0.6,Y[hitY]),NW,Fill(white+opacity(0.7)));
label("$Y_{\tau^*}$",(T-hitY+0.6,Y[hitY]),NW);

label("\phantom{$\rho^{Root}$}",(hitX,-0.6),S);

// BRACE
// pen brc = black + 0.7;
// draw(brace((T,X[T-t]), (T, Y[t-1])), brc);
// label("$X_{T-t} - Y_{t-1}$", (T+1,(X[T-t]+Y[t-1])/2),E);
\end{asy}

%% file: SSRWPlot1.tex
\begin{asy}[width=1\textwidth]

import graph;
import stats;
import patterns;

defaultpen(fontsize(8pt));

// define boundary points
int T = 14;
int xmax = 8;
int t = 6;
int hitX = T-t;

// resp. starting points
int x = 4;
int y = 3;

// grid 
pen grid = mediumgray + 0.2;

/*
for(int i=1; i <= xmax; ++i){
draw((0,i)--(T,i), grid);
}

for(int i=1; i <= T; ++i){
draw((i,0)--(i,xmax), grid);
}
*/

// MANUAL GRID
// horizontal lines

draw((0,1)--(2,1), grid);
draw((0,2)--(11,2), grid);
draw((0,3)--(T,3), grid);
draw((0,4)--(T,4), grid);
draw((0,5)--(6,5), grid);
draw((0,6)--(6,6), grid);
draw((0,7)--(3,7), grid);

//vertical lines

draw((1,1)--(1,7), grid);
draw((2,1)--(2,7), grid);
draw((3,2)--(3,7), grid);
draw((4,2)--(4,6), grid);
draw((5,2)--(5,6), grid);
draw((6,2)--(6,6), grid);
draw((7,2)--(7,4), grid);
draw((8,2)--(8,4), grid);
draw((9,2)--(9,4), grid);
draw((10,2)--(10,4), grid);
draw((11,2)--(11,4), grid);
draw((12,3)--(12,4), grid);
draw((13,3)--(13,4), grid);
draw((14,3)--(14,4), grid);

// DRAW A BARRIER
pen barrier = orange + 2;

// horizontal lines
draw((1,8)--(T,8),barrier);
draw((4,7)--(T,7),barrier);
draw((7,6)--(T,6),barrier);
draw((7,5)--(T,5),barrier);
draw((12,2)--(T,2),barrier);
//draw((13,2)--(T,2),barrier);
draw((3,1)--(T,1),barrier);
draw((0,0)--(T,0),barrier);

/*
// vertical lines
draw((3,0)--(3,1), barrier);
draw((4,0)--(4,1), barrier);
draw((5,0)--(5,1), barrier);
draw((6,0)--(6,1), barrier);
draw((7,0)--(7,1), barrier);
draw((8,0)--(8,1), barrier);
draw((9,0)--(9,1), barrier);
draw((10,0)--(10,1), barrier);
draw((11,0)--(11,1), barrier);
draw((12,0)--(12,2), barrier);
draw((13,0)--(13,2), barrier);
draw((14,0)--(14,2), barrier);
*/
 
// horizontal lines
pen ph = black + 0.7 + dashed;

//draw((T-t,0)--(T-t,5), heavyred + 0.5 + dashed);
draw((T-t,0 - 0.05)--(T-t,5), ph);
//draw((hitX,0 - 0.05)--(hitX,1), ph);
draw((T-t+1,0-0.05)--(T-t+1,4), ph);

// X-path
real X0 = x;
real[] X; // random walk
X[0] = X0;
path RWX;
RWX = (0,X[0]);

// Get a path:
/* Random Path Generation
real[] pathX; // path 
srand(15); //seed for X
for(int i=1; i <= T-t; ++i){
    real coin = Gaussrand();
    if(coin > 0){
        pathX[i] = 1 ;
    }
    else{
        pathX[i] = -1;
    }
}
*/ 

// Get path manually:
real[] pathX = {1,-1,-1,-1,1,1,-1,1};

for(int i=1; i <= hitX; ++i){
    X[i] = X[i-1] + pathX[i-1];
    RWX = RWX--(i,X[i]);
}

// Y-path
real Y0 = y;
real[] Y; // random walk
Y[0] = Y0;
path RWY;
RWY = (T,Y[0]);

// Get a Y-path:
/* Random Path Generation
real[] pathY; // path 
srand(9); //seed for Y
for(int i=1; i <= t-1; ++i){
    real coin = Gaussrand();
    if(coin > 0){
        pathY[i] = 1 ;
    }
    else{
        pathY[i] = -1;
    }
}
*/ 

// Get path manually:
real[] pathY = {1,-1,-1,1,1};

for(int i=1; i <= t-1; ++i){
    Y[i] = Y[i-1] + pathY[i-1];
    RWY = RWY--(T-i,Y[i]);
}

// pens for graphs
pen pX = heavyblue + 1.5;
pen pY = heavyred + 1.5;
pen pYY = heavyred + 1.5 + linetype(new real[] {4,3});

// draw random paths
draw(RWX,pX);
draw(RWY,pY);

draw((T-t+1,Y[t-1])--(T-t,Y[t-1]+1),pYY);
draw((T-t+1,Y[t-1])--(T-t,Y[t-1]-1),pYY);

// AXIS
pen axis = black + 0.5;
draw((0,0)--(0,xmax+0.6),axis,Arrow);
draw((0,0)--((T+0.6),0),axis,Arrow);

// plot starting dots;
dot((0,x),heavyblue);
dot((T,y),heavyred);

// plot end dots;
dot((hitX,X[hitX]),heavyblue);
dot((T-t+1,Y[t-1]),heavyred);
dot((T-t,Y[t-1]+1),heavyred);
dot((T-t,Y[t-1]-1),heavyred);

// TIME T TICK
draw((T,-0.05)--(T,0.05), axis);

// LABELS
label("$T$",(T,0),S);
//label("$T$-$s$",(T-t,0),S);
label("$T$-$(s$-$1)$",(T-t+1,-0.6),S);

label("$\sigma_s = T$-$s$",(hitX,0),S);

label("$x$",(0,x),W);
label("$y$",(T,y),E);

label("$X_{\sigma_s}$",(hitX,X[hitX]),W);
label("$Y_{s-1}$",(T-t+1,Y[t-1]),E);

// BRACE
// pen brc = black + 0.7;
// draw(brace((T,X[T-t]), (T, Y[t-1])), brc);
// label("$X_{T-t} - Y_{t-1}$", (T+1,(X[T-t]+Y[t-1])/2),E);
\end{asy}

%% file: SSRWPlot2.tex
\begin{asy}[width=1\textwidth]

import graph;
import stats;
import patterns;

defaultpen(fontsize(8pt));

// define boundary points
int T = 14;
int xmax = 8;
int t = 6;
int hitX = 5;

// resp. starting points
int x = 4;
int y = 3;

// grid 
pen grid = mediumgray + 0.2;

/*
for(int i=1; i <= xmax; ++i){
draw((0,i)--(T,i), grid);
}

for(int i=1; i <= T; ++i){
draw((i,0)--(i,xmax), grid);
}
*/

// MANUAL GRID
// horizontal lines

draw((0,1)--(2,1), grid);
draw((0,2)--(11,2), grid);
draw((0,3)--(T,3), grid);
draw((0,4)--(T,4), grid);
draw((0,5)--(6,5), grid);
draw((0,6)--(6,6), grid);
draw((0,7)--(3,7), grid);

//vertical lines

draw((1,1)--(1,7), grid);
draw((2,1)--(2,7), grid);
draw((3,2)--(3,7), grid);
draw((4,2)--(4,6), grid);
draw((5,2)--(5,6), grid);
draw((6,2)--(6,6), grid);
draw((7,2)--(7,4), grid);
draw((8,2)--(8,4), grid);
draw((9,2)--(9,4), grid);
draw((10,2)--(10,4), grid);
draw((11,2)--(11,4), grid);
draw((12,3)--(12,4), grid);
draw((13,3)--(13,4), grid);
draw((14,3)--(14,4), grid);

// DRAW A BARRIER
pen barrier = orange + 2;

// horizontal lines
draw((1,8)--(T,8),barrier);
draw((4,7)--(T,7),barrier);
draw((7,6)--(T,6),barrier);
draw((7,5)--(T,5),barrier);
draw((12,2)--(T,2),barrier);
//draw((13,2)--(T,2),barrier);
draw((3,1)--(T,1),barrier);
draw((0,0)--(T,0),barrier);

/*
// vertical lines
draw((3,0)--(3,1), barrier);
draw((4,0)--(4,1), barrier);
draw((5,0)--(5,1), barrier);
draw((6,0)--(6,1), barrier);
draw((7,0)--(7,1), barrier);
draw((8,0)--(8,1), barrier);
draw((9,0)--(9,1), barrier);
draw((10,0)--(10,1), barrier);
draw((11,0)--(11,1), barrier);
draw((12,0)--(12,2), barrier);
draw((13,0)--(13,2), barrier);
draw((14,0)--(14,2), barrier);
*/

// DRAW THE BARRIER (SOLID)
//fill((3,0)--(3,1)--(T-2,1)--(T-2,2)--(T,2)--(T,0)--cycle, black);

// horizontal lines
pen ph = black + 0.7 + dashed;

//draw((T-t,0)--(T-t,5), heavyred + 0.5 + dashed);
draw((T-t,0 - 0.05)--(T-t,5), ph);
draw((hitX,0 - 0.05)--(hitX,1), ph);
draw((T-t+1,0-0.05)--(T-t+1,4), ph);

// X-path
real X0 = x;
real[] X; // random walk
X[0] = X0;
path RWX;
RWX = (0,X[0]);

// Get a path:
/* Random Path Generation
real[] pathX; // path 
srand(15); //seed for X
for(int i=1; i <= T-t; ++i){
    real coin = Gaussrand();
    if(coin > 0){
        pathX[i] = 1 ;
    }
    else{
        pathX[i] = -1;
    }
}
*/ 

// Get path manually:
real[] pathX = {-1,1,-1,-1,-1};

for(int i=1; i <= hitX; ++i){
    X[i] = X[i-1] + pathX[i-1];
    RWX = RWX--(i,X[i]);
}

// Y-path
real Y0 = y;
real[] Y; // random walk
Y[0] = Y0;
path RWY;
RWY = (T,Y[0]);

// Get a Y-path:
/* Random Path Generation
real[] pathY; // path 
srand(9); //seed for Y
for(int i=1; i <= t-1; ++i){
    real coin = Gaussrand();
    if(coin > 0){
        pathY[i] = 1 ;
    }
    else{
        pathY[i] = -1;
    }
}
*/ 

// Get path manually:
real[] pathY = {1,-1,-1,1,1};

for(int i=1; i <= t-1; ++i){
    Y[i] = Y[i-1] + pathY[i-1];
    RWY = RWY--(T-i,Y[i]);
}

// pens for graphs
pen pX = heavyblue + 1.5;
pen pY = heavyred + 1.5;
pen pYY = heavyred + 1.5 + linetype(new real[] {4,3});

// draw random paths
draw(RWX,pX);
draw(RWY,pY);

draw((T-t+1,Y[t-1])--(T-t,Y[t-1]+1),pYY);
draw((T-t+1,Y[t-1])--(T-t,Y[t-1]-1),pYY);

// AXIS
pen axis = black + 0.5;
draw((0,0)--(0,xmax+0.6),axis,Arrow);
draw((0,0)--((T+0.6),0),axis,Arrow);

// plot starting dots;
dot((0,x),heavyblue);
dot((T,y),heavyred);

// plot end dots;
dot((hitX,X[hitX]),heavyblue);
dot((T-t+1,Y[t-1]),heavyred);
dot((T-t,Y[t-1]+1),heavyred);
dot((T-t,Y[t-1]-1),heavyred);

// TIME T TICK
draw((T,-0.05)--(T,0.05), axis);

// LABELS
label("$T$",(T,0),S);
label("$T$-$s$",(T-t,0),S);
label("$T$-$(s$-$1)$",(T-t+1,-0.6),S);

label("$\sigma_s$",(hitX,0),S);

label("$x$",(0,x),W);
label("$y$",(T,y),E);

label("$X_{\sigma_s}$",(hitX,X[hitX]),NE);
label("$Y_{s-1}$",(T-t+1,Y[t-1]),E);

// BRACE
// pen brc = black + 0.7;
// draw(brace((T,X[T-t]), (T, Y[t-1])), brc);
// label("$X_{T-t} - Y_{t-1}$", (T+1,(X[T-t]+Y[t-1])/2),E);
\end{asy}

%% file: RootToRost1.tex
\begin{asy}[width=1\textwidth]
import graph;
import stats;
import patterns;

defaultpen(fontsize(8pt));

    // define boundary points
int Tmax = 14;
int T = 11;
int xmax = 8;
int t = 6;
int hitX = 8;
int hitY = 7;
real tick = 0.05;

    // resp. starting points
int x = 0;
int y = -1;

    // grid 
pen grid = mediumgray + 0.2;

    // MANUAL GRID
    // horizontal lines
draw((7,3)--(Tmax,3), grid);
draw((4,2)--(Tmax,2), grid);
draw((0,1)--(Tmax,1), grid);
draw((0,-1)--(Tmax,-1), grid);
draw((2,-2)--(Tmax,-2), grid);
draw((4,-3)--(Tmax,-3), grid);
draw((9,-4)--(Tmax,-4), grid);

    //vertical lines
draw((1,1)--(1,-1), grid);
draw((2,1)--(2,-2), grid);
draw((3,1)--(3,-2), grid);
draw((4,2)--(4,-3), grid);
draw((5,2)--(5,-3), grid);
draw((6,2)--(6,-3), grid);
draw((7,3)--(7,-3), grid);
draw((8,3)--(8,-3), grid);
draw((9,3)--(9,-4), grid);
draw((10,3)--(10,-4), grid);
draw((11,3)--(11,-4), grid);
draw((12,3)--(12,-4), grid);
draw((13,3)--(13,-4), grid);
draw((14,3)--(14,-4), grid);

    // DRAW A BARRIER
pen barrier = orange + 2;

    // horizontal lines
draw((0,4)--(14,4),barrier);
draw((0,3)--(6,3),barrier);
draw((0,2)--(3,2),barrier);
draw((0,-2)--(1,-2),barrier);
draw((0,-3)--(3,-3),barrier);
draw((0,-4)--(8,-4),barrier);
 
    // AXIS
pen axis = black + 0.5;
real arr = 0.8;

draw((0,0)--(0,4+arr),axis,Arrow);
draw((0,0)--(0,-4-arr),axis,Arrow);
draw((0,0)--((Tmax+arr),0),axis,Arrow);
draw((T,4)--(T,-4-tick), black + 1 + linetype(new real[] {4,3}));

    // horizontal lines
pen ph = black + 0.7 + dashed;

    // X-path
real X0 = x;
real[] X; // random walk
X[0] = X0;
path RWX;
RWX = (0,X[0]);

    // Get an X-path:
    // Get path manually:
real[] pathX = {-1,1,1,1,-1,1,1,1};

for(int i=1; i <= hitX; ++i){
    X[i] = X[i-1] + pathX[i-1];
    RWX = RWX--(i,X[i]);
} 

    // Y-path
real Y0 = y;
real[] Y; // random walk
Y[0] = Y0;
path RWY;
RWY = (T,Y[0]);

    // Get a Y-path:
    // Get path manually:
real[] pathY = {-1,-1,1,-1,1,-1,-1};

for(int i=1; i <= hitY; ++i){
    Y[i] = Y[i-1] + pathY[i-1];
    RWY = RWY--(T-i,Y[i]);
}

    // Pens for graphs
pen pX = heavyblue + 1.5;
pen pY = heavyred + 1.5;
pen pYY = heavyred + 1.5 + linetype(new real[] {4,3});

    // Draw paths
draw(RWX,pX);
draw(RWY,pY);

    // plot starting dots;
dot((0,x),heavyblue);
dot((T,y),heavyred);

    // plot end dots;
dot((hitX,X[hitX]),heavyblue);
dot((T-hitY,Y[hitY]),heavyred);

    // LABELS
draw((hitX,4)--(hitX,0 - tick), heavyblue + 0.5 + dashed);

label("$x$",(0,x),NE, Fill(white+opacity(0.7)));
label("$y$",(T,y),E, Fill(white+opacity(0.7)));

//label("$\rho^{Rost}$",(hitX,0),N,heavyblue,UnFill(0mm));
label("$\rho^{Rost}$",(hitX,0),N,heavyblue,Fill(white+opacity(0.7)));

draw((T-hitY,-4)--(T-hitY,0 + tick), heavyred + 0.5 + dashed);

//label("$T-\rho^{Root}$",(T-hitY,0),S,heavyred, UnFill(0mm));
label("$T-\rho^{Root}$",(T-hitY,0),S,heavyred, Fill(white+opacity(0.7)));

label("$T$",(T,-4),S);

label("$D^{Rost}$", (Tmax, 4), NW, orange);
\end{asy}

%% file: RootToRost2.tex
\begin{asy}[width=1\textwidth]
import graph;
import stats;
import patterns;

defaultpen(fontsize(8pt));

    // define boundary points
int Tmax = 14;
int T = 11;
int xmax = 8;
int t = 2;
int hitX = 7;
int hitY = 8;
real tick = 0.05;

    // resp. starting points
int x = -1;
int y = 0;
 
    // grid 
pen grid = mediumgray + 0.2;

    // MANUAL GRID
    // horizontal lines

draw((-3,3)--(4,3), grid);
draw((-3,2)--(7,2), grid);
draw((-3,1)--(T,1), grid);
draw((-3,0)--(T,0), grid);
draw((-3,-1)--(T,-1), grid);
draw((-3,-2)--(9,-2), grid);
draw((-3,-3)--(7,-3), grid);
draw((-3,-4)--(2,-4), grid);

    //vertical lines

draw((-3,3)--(-3,-4), grid);
draw((-2,3)--(-2,-4), grid);
draw((-1,3)--(-1,-4), grid);
draw((1,3)--(1,-4), grid);
draw((2,3)--(2,-4), grid);
draw((3,3)--(3,-3), grid);
draw((4,3)--(4,-3), grid);
draw((5,2)--(5,-3), grid);
draw((6,2)--(6,-3), grid);
draw((7,2)--(7,-3), grid);
draw((8,1)--(8,-2), grid);
draw((9,1)--(9,-2), grid);
draw((10,1)--(10,-1), grid);
draw((11,1)--(11,-1), grid);

    // DRAW A BARRIER
pen barrier = orange + 2;

    // horizontal lines
draw((-3,4)--(T,4),barrier);
draw((5,3)--(T,3),barrier);
draw((8,2)--(T,2),barrier);
draw((10,-2)--(T,-2),barrier);
draw((8,-3)--(T,-3),barrier);
draw((3,-4)--(T,-4),barrier);

    // AXIS
pen axis = black + 0.5;
real arr = 0.8;

draw((0,0)--(0,4+arr),axis,Arrow);
draw((0,0)--(0,-4-arr),axis,Arrow);
draw((0,0)--((T+arr),0),axis,Arrow);
draw((T,4)--(T,-4-tick), black + 1 + linetype(new real[] {4,3}));

    // X-path
real X0 = x;
real[] X; // random walk
X[0] = X0;
path RWX;
RWX = (0,X[0]);

    // Get an X-path:
    // Get path manually:
real[] pathX = {-1,-1,1,-1,1,-1,-1};

for(int i=1; i <= hitX; ++i){
    X[i] = X[i-1] + pathX[i-1];
    RWX = RWX--(i,X[i]);
}

    // Y-path
real Y0 = y;
real[] Y; // random walk
Y[0] = Y0;
path RWY;
RWY = (T,Y[0]);

    // Get a Y-path:
    // Get path manually:
real[] pathY = {-1,1,1,1,-1,1,1,1};

for(int i=1; i <= hitY; ++i){
    Y[i] = Y[i-1] + pathY[i-1];
    RWY = RWY--(T-i,Y[i]);
}

    // pens for graphs
pen pX = heavyblue + 1.5;
pen pY = heavyred + 1.5;
pen pYY = heavyred + 1.5 + linetype(new real[] {4,3});

    // draw random paths
draw(RWX,pY);
draw(RWY,pX);

    // plot starting dots;
dot((0,x),heavyred);
dot((T,y),heavyblue);

    // plot end dots;
dot((hitX,X[hitX]),heavyred);
dot((T-hitY,Y[hitY]),heavyblue);

    // LABELS

label("$x$",(0,x),W, Fill(white+opacity(0.7)));
label("$y$",(T,y),NW, Fill(white+opacity(0.7)));

draw((hitX,-4)--(hitX,0 + tick), heavyred + 0.5 + dashed);
//label("$\rho^{Root}$",(hitX,0),S,heavyred, UnFill(0mm));
label("$\rho^{Root}$",(hitX,0),S,heavyred, Fill(white+opacity(0.7)));

draw((T-hitY,4)--(T-hitY,0 - tick), heavyblue + 0.5 + dashed);
//label("$T-\rho^{Rost}$",(T-hitY,0),N,heavyblue, UnFill(0mm));
label("$T-\rho^{Rost}$",(T-hitY,0),N,heavyblue, Fill(white+opacity(0.7)));

label("$T$",(T,-4),S);

label("$D$", (T, 4), NW, orange);
\end{asy}

%% file: SwitchingIdentities.bbl
\begin{thebibliography}{10}

\bibitem{BeCoHu17}
M.~Beiglb{\"o}ck, A.~M.~G. Cox, and M.~Huesmann.
\newblock Optimal transport and skorokhod embedding.
\newblock {\em Inventiones mathematicae}, 208(2):327--400, May 2017.

\bibitem{Ch77}
R.~V. Chacon.
\newblock Potential processes.
\newblock {\em Transactions of the American Mathematical Society}, 226:39--58,
  1977.

\bibitem{CoObTo18}
A.~M.~G. {Cox}, J.~{Ob{\l}{\'o}j}, and N.~{Touzi}.
\newblock The root solution to the multi-marginal embedding problem: an optimal
  stopping and time-reversal approach.
\newblock {\em Probability Theory and Related Fields}, Feb 2018.

\bibitem{CoWa13}
A.~M.~G. {Cox} and J.~{Wang}.
\newblock {Optimal robust bounds for variance options}.
\newblock {\em Preprint arXiv:1308.4363}, 2013.

\bibitem{CoWa12}
A.~M.~G. {Cox} and J.~{Wang}.
\newblock {Root's Barrier: Construction, Optimality and Applications to
  Variance Options}.
\newblock {\em Ann. Appl. Probab.}, 23(3):859--894, 2013.

\bibitem{CoKi19}
A.~M.G. {Cox} and S.~M. {Kinsley}.
\newblock Discretisation and duality of optimal skorokhod embedding problems.
\newblock {\em Stochastic Processes and their Applications}, 129(7):2376 --
  2405, 2019.

\bibitem{DA18}
T.~De~Angelis.
\newblock From optimal stopping boundaries to rost’s reversed barriers and
  the skorokhod embedding.
\newblock {\em Ann. Inst. H. Poincaré Probab. Statist.}, 54(2):1098--1133, 05
  2018.

\bibitem{Gr17}
A.~M. {Grass}.
\newblock Uniqueness properties of barrier type skorokhod embeddings and
  perkins embedding with general starting law.
\newblock {\em Master Thesis}, 2017.
\newblock available online at https://www.mat.univie.ac.at/$\sim$grass/.

\bibitem{Ho11}
D.~Hobson.
\newblock The {S}korokhod embedding problem and model-independent bounds for
  option prices.
\newblock In {\em Paris-{P}rinceton {L}ectures on {M}athematical {F}inance
  2010}, volume 2003 of {\em Lecture Notes in Math.}, pages 267--318. Springer,
  Berlin, 2011.

\bibitem{Lawler_intersections}
G.~F. Lawler.
\newblock {\em Intersections of random walks}.
\newblock Modern Birkh\"auser Classics. Birkh\"auser/Springer, New York, 2013.
\newblock Reprint of the 1996 edition.

\bibitem{Lo70}
R.~M. Loynes.
\newblock Stopping times on {B}rownian motion: {S}ome properties of {R}oot's
  construction.
\newblock {\em Z. Wahrscheinlichkeitstheorie und Verw. Gebiete}, 16:211--218,
  1970.

\bibitem{MC91}
T.~R. McConnell.
\newblock The two-sided stefan problem with a spatially dependent latent heat.
\newblock {\em Transactions of the American Mathematical Society},
  326(2):669--699, 1991.

\bibitem{Ob04}
J.~Ob{\l}{\'o}j.
\newblock The {S}korokhod embedding problem and its offspring.
\newblock {\em Probab. Surv.}, 1:321--390, 2004.

\bibitem{GaObZo19}
H.~{Oberhauser} P.~{Gassiat} and C.~Z. {Zou}.
\newblock A free boundary characterization of the root barrier for markov
  processes.
\newblock {\em Preprint arXiv:1905.13174}, 2019.

\bibitem{Ro69}
D.~H. Root.
\newblock The existence of certain stopping times on {B}rownian motion.
\newblock {\em Ann. Math. Statist.}, 40:715--718, 1969.

\bibitem{Ro76}
H.~Rost.
\newblock Skorokhod stopping times of minimal variance.
\newblock In {\em S\'eminaire de {P}robabilit\'es, {X} ({P}remi\`ere partie,
  {U}niv. {S}trasbourg, {S}trasbourg, ann\'ee universitaire 1974/1975)}, pages
  194--208. Lecture Notes in Math., Vol. 511. Springer, Berlin, 1976.

\bibitem{Sk61}
A.~V. Skorohod.
\newblock {\em Issledovaniya po teorii sluchainykh protsessov
  ({S}tokhasticheskie differentsialnye uravneniya i predelnye teoremy dlya
  protsessov {M}arkova)}.
\newblock Izdat. Kiev. Univ., Kiev, 1961.

\bibitem{Sk65}
A.~V. Skorokhod.
\newblock {\em Studies in the theory of random processes}.
\newblock Translated from the Russian by Scripta Technica, Inc. Addison-Wesley
  Publishing Co., Inc., Reading, Mass., 1965.

\end{thebibliography}
